\documentclass[12pt]{article}
\pdfoutput=1

\usepackage[affil-it]{authblk}
\usepackage{amsmath, amsthm, amssymb}
\usepackage{fullpage}
\usepackage{enumerate}
\usepackage{ifpdf}
\usepackage{mathrsfs}
\usepackage[normalem]{ulem}
\ifpdf
\usepackage[pdftex]{graphicx}
\else
\usepackage[dvips]{graphicx}
\fi
\usepackage{tikz}
\usepackage[all]{xy}
\usetikzlibrary{arrows,chains,matrix,positioning,scopes}

\usepackage{tkz-graph}
\GraphInit[vstyle = Classic]
\tikzset{
  VertexStyle/.append style = {minimum size=5pt, inner sep=0pt,
                                font = \Large\bfseries},
  EdgeStyle/.append style = {} }

\makeatletter
\tikzset{join/.code=\tikzset{after node path={%
\ifx\tikzchainprevious\pgfutil@empty\else(\tikzchainprevious)%
edge[every join]#1(\tikzchaincurrent)\fi}}}
\makeatother
\tikzset{>=stealth',every on chain/.append style={join},
         every join/.style={->}}
\tikzstyle{labeled}=[execute at begin node=$\scriptstyle,
   execute at end node=$]

\input xy
\xyoption{all}

\usepackage[pdftex,plainpages=false,hypertexnames=false,pdfpagelabels]{hyperref}
\newcommand{\arXiv}[1]{\href{http://arxiv.org/abs/#1}{\tt arXiv:\nolinkurl{#1}}}
\newcommand{\arxiv}[1]{\href{http://arxiv.org/abs/#1}{\tt arXiv:\nolinkurl{#1}}}

\newcommand{\doi}[1]{\href{http://dx.doi.org/#1}{{\tt DOI:#1}}}

\newcommand{\googlebooks}[1]{(preview at \href{http://books.google.com/books?id=#1}{google books})}

\usepackage{xcolor}
\definecolor{dark-red}{rgb}{0.7,0.25,0.25}
\definecolor{dark-blue}{rgb}{0.15,0.15,0.55}
\definecolor{medium-blue}{rgb}{0,0,.8}
\definecolor{DarkGreen}{RGB}{0,150,0}
\hypersetup{
   colorlinks, linkcolor={purple},
   citecolor={medium-blue}, urlcolor={medium-blue}
}

\theoremstyle{plain}
\newtheorem{thm}{Theorem}[section]
\newtheorem*{thm*}{Theorem}

\newtheorem{cor}[thm]{Corollary}
\newtheorem*{cor*}{Corollary}
\newtheorem{lem}[thm]{Lemma}

\newtheorem{prop}[thm]{Proposition}

\newtheorem*{quest*}{Question}

\theoremstyle{definition}
\newtheorem{defn}[thm]{Definition}

\newtheorem{assumptions}[thm]{Assumptions}

\newtheorem{rem}[thm]{Remark}

\newtheorem{fact}[thm]{Fact}

\DeclareMathOperator{\op}{op}

\DeclareMathOperator{\Tr}{Tr}
\DeclareMathOperator{\tr}{tr}

\newcommand{\comment}[1]{}

\newcommand{\be}{\begin{enumerate}[(1)]}
\newcommand{\ee}{\end{enumerate}}

\newcommand{\N}{\mathbb{N}}
\newcommand{\Z}{\mathbb{Z}}

\newcommand{\F}{\mathbb{F}}
\newcommand{\R}{\mathbb{R}}

\newcommand{\C}{\mathbb{C}}
\newcommand{\B}{\mathbb{B}}

\newcommand{\J}{\mathscr{J}}

\newcommand{\set}[2]{\left\{#1 \middle| #2\right\}}

\newcommand{\e}{\epsilon}

\newcommand{\noshow}[1]{}
\newcommand{\MR}[1]{}

\newcommand{\p}{\partial}

\newcommand{\TL}{\cT\hspace{-.08cm}\cL}

\newcommand{\D}{\mathscr{D}}

\def\semicolon{;}
\def\applytolist#1{
    \expandafter\def\csname multi#1\endcsname##1{
        \def\multiack{##1}\ifx\multiack\semicolon
            \def\next{\relax}
        \else
            \csname #1\endcsname{##1}
            \def\next{\csname multi#1\endcsname}
        \fi
        \next}
    \csname multi#1\endcsname}

\def\calc#1{\expandafter\def\csname c#1\endcsname{{\mathcal #1}}}
\applytolist{calc}QWERTYUIOPLKJHGFDSAZXCVBNM;
\def\bbc#1{\expandafter\def\csname bb#1\endcsname{{\mathbb #1}}}
\applytolist{bbc}QWERTYUIOPLKJHGFDSAZXCVBNM;
\def\bfc#1{\expandafter\def\csname bf#1\endcsname{{\mathbf #1}}}
\applytolist{bfc}QWERTYUIOPLKJHGFDSAZXCVBNM;
\def\sfc#1{\expandafter\def\csname s#1\endcsname{{\sf #1}}}
\applytolist{sfc}QWERTYUIOPLKJHGFDSAZXCVBNM;
\def\ffc#1{\expandafter\def\csname f#1\endcsname{{\mathfrak #1}}}
\applytolist{ffc}QWERTYUIOPLKJHGFDSAZXCVBNM;

\usepackage{yfonts}

\tikzstyle{shaded}=[fill=red!10!blue!20!gray!30!white]
\tikzstyle{unshaded}=[fill=white]
\tikzstyle{empty box}=[circle, draw, thick, fill=white, opaque, inner sep=2mm]
\tikzstyle{annular}=[scale=.7, inner sep=1mm, baseline]
\tikzstyle{rectangular}=[scale=.75, inner sep=1mm, baseline=-.1cm]
\usetikzlibrary{calc}

\newcommand{\nbox}[6]{
	\draw[thick, #1] ($#2+(-#3,-#3)+(-#4,0)$) rectangle ($#2+(#3,#3)+(#5,0)$);
	\coordinate (ZZa) at ($#2+(-#4,0)$);
	\coordinate (ZZb) at ($#2+(#5,0)$);
	\node at ($1/2*(ZZa)+1/2*(ZZb)$) {#6};
}

\begin{document}

\title{Free transport for interpolated free group factors}
\author{Michael Hartglass and Brent Nelson}
\date{}
\maketitle

\abstract{In this article, we study a form of free transport for the interpolated free group factors, extending the work of Guionnet and Shlyakhtenko for the usual free group factors \cite{MR3251831}.  Our model for the interpolated free group factors comes from a canonical finite von Neumann algebra $\cM(\Gamma, \mu)$ associated to a finite, connected, weighted graph $(\Gamma,V,E, \mu)$ \cite{MR3110503,1509.0255}.  With this model, we use an operator-valued version of Voiculescu's free difference quotient introduced in \cite{1509.0255} to state a Schwinger--Dyson equation which is valid for the generators of $\cM(\Gamma, \mu)$.  We construct free transport for appropriate perturbations of this equation. Also, $\cM(\Gamma, \mu)$ can be constructed using the machinery of Shlyakhtenko's operator-valued semicircular systems \cite{MR1704661}.}

\section*{Introduction}

The interpolated free group factors $L(\F_{t})$ for $t \in [1, \infty]$ were discovered and developed independently by Dykema \cite{MR1256179} and R\v{a}dulescu \cite{MR1258909}.  They satisfy the following properties:
\begin{itemize}

\item $L(\F_{t}) * L(\F_{s}) = L(\F_{s+t})$

\item $pL(\F_{t})p = \L(\F_{r})$ where $r = 1 + \frac{t-1}{\tr(p)^{2}}$ and $p$ is a nonzero projection in $L(\F_{t})$

\item If $t \in \N \cup \{\infty\}$ then $L(\F_{t})$ is the usual free group factor on $t$ generators.

\end{itemize}
For non-integer $t$, these share many of the same properties of their integer counterparts.  Namely, they are non-$\Gamma$, strongly solid II$_{1}$ factors.  In this paper, we demonstrate another similarity: the existence of free transport.

For our purposes, the most convenient way to describe an interpolated free group factor is via a weighted graph.     Specifically, we consider a finite, connected, undirected, weighted graph $(\Gamma, V, E, \mu)$ with vertex set $V$, edge set $E$, and weighting $\mu: V \rightarrow (0, 1]$ satisfying $\sum_{v \in V} \mu(v) = 1$.  One can associate to this data a C$^{*}$-algebra $\cS(\Gamma, \mu)$ and a von Neumann algebra $\cM(\Gamma, \mu)$.  A simple non-degeneracy condition on the weighting determines whether $\cM(\Gamma, \mu)$ is a factor, and when $\cS(\Gamma, \mu)$ is simple with unique trace.  If $\cM(\Gamma, \mu)$ is a factor, then it is necessarily isomorphic to $\L(\F_{t})$ where 
$$
t = 1 - \sum_{v \in V} \mu(v)^{2} + \sum_{v \in V} \mu(v) \sum_{w \sim v} n_{v, w}\mu(w).
$$
See Equation \ref{eqn:parameter} and the discussion immediately preceding it.  In particular, if $\Gamma$ consists of a single vertex with  $n$-loops, $\cM(\Gamma, \mu)\cong L(\F_{n})$, and $\cS(\Gamma, \mu)$ is the C*-algebra generated by a free semicircular system (\textit{cf.} Figure \ref{fig:Graph}).  The algebras $\cM(\Gamma, \mu)$ were initially studied in \cite{MR2807103} in determining the isomorphism classes of von Neumann algebras arising from planar algebras.  Slightly more general versions of $\cM(\Gamma, \mu)$ were studied by the first author in \cite{MR3110503}, and the C*-algebra counterparts $\cS(\Gamma, \mu)$ were studied in \cite{1401.2485,MR3266249,1509.0255}.

\begin{figure}[!htb]
\begin{center}
\caption{von Neumann algebras corresponding to simple graphs}\label{fig:Graph}
\begin{tabular}{|c|c|}
\hline
$\Gamma$ & $\cM(\Gamma, \mu)$\\
\hline
\begin{tikzpicture}[baseline = -.1cm]
  \SetGraphUnit{1}
 \Vertex[Lpos=0]{1}
  \Loop[dist = 2cm, dir = NO](1.north)
  \Loop[dist = 2cm, dir = SO](1.south)
  \tikzset{EdgeStyle/.append style = {}}
 \draw [dashed] (-.5, .5) arc (135:225:.71cm);
 \node at (-1, 0) {\scriptsize{$n$}};
\end{tikzpicture}

& \begin{tikzpicture}[baseline = -.1cm]
\node at (0, 0){$L(\F_{n})$};
\end{tikzpicture}\\
\hline

\begin{tikzpicture}[baseline = -.1cm]
  \SetGraphUnit{1}
 \Vertex[Lpos=5]{a}
\EA(a){1-a}
 
  \Edge(a)(1-a)
  \Loop[dist = 2cm, dir = NO](a.north)
  \Loop[dist = 2cm, dir = SO](a.south)
  \tikzset{EdgeStyle/.append style = {}}
 \draw [dashed] (-.5, .5) arc (135:225:.71cm);
 \node at (-1, 0) {\scriptsize{$n$}};
 \node at (1.5, -1) {\scriptsize{$a \in [\frac{1}{2}, 1)$}};
\end{tikzpicture}

& \begin{tikzpicture}[baseline = -.1cm]
\node at (0, .5){$L(\F_{t})$};
\node at (0, -.5){$t = (n-4)a^2 +4a$};
\end{tikzpicture}\\
\hline

\begin{tikzpicture}[baseline = -.1cm]
  \SetGraphUnit{1}
 \Vertex[Lpos=0]{a}
\EA(a){1-a}
\tikzset{EdgeStyle/.append style = {bend left = 50}} 
  \Edge(a)(1-a)
  \tikzset{EdgeStyle/.append style = {bend right = 50}} 
  \Edge(a)(1-a)
 \node at (1.5, -.7) {\scriptsize{$a \in [\frac{1}{3}, \frac{2}{3}]$}};
\end{tikzpicture}
&
\begin{tikzpicture}[baseline = -.1cm]
\node at (0, .5){$L(\F_{t})$};
\node at (0, -.5){$t = 6(a - a^{2})$};
\end{tikzpicture}\\
\hline

\begin{tikzpicture}[baseline = -.1cm]
  \SetGraphUnit{1}
 \Vertex[Lpos=0]{a}
\EA(a){1-a}
\Loop[dist = 2cm, dir = NO](a.north)
  \Loop[dist = 2cm, dir = SO](a.south)
\tikzset{EdgeStyle/.append style = {bend left = 50}} 
  \Edge(a)(1-a)
  \tikzset{EdgeStyle/.append style = {bend right = 50}} 
  \Edge(a)(1-a)
 \draw [dashed] (-.5, .5) arc (135:225:.71cm);
 \node at (-1, 0) {\scriptsize{$n$}};
 \node at (1.5, -.7) {\scriptsize{$a \in [\frac{1}{3}, 1)$}};
 \node at (1.5, -1.2) {\scriptsize{$n \geq 1$}};
\end{tikzpicture}
&
\begin{tikzpicture}[baseline = -.1cm]
\node at (0, .5){$L(\F_{t})$};
\node at (0, -.5){$t = (n-6)a^2 +6a$};
\end{tikzpicture}\\
\hline

\begin{tikzpicture}[baseline = -.1cm]
  \SetGraphUnit{1}
 \Vertex[Lpos=5]{a}
\EA(a){b}
\NO(b){1-a-b}
\Loop[dist = 2cm, dir = NO](a.north)
  \Loop[dist = 2cm, dir = SO](a.south)
  \Edge(a)(b)
  \Edge(b)(1-a-b)
 \draw [dashed] (-.5, .5) arc (135:225:.71cm);
 \node at (-1, 0) {\scriptsize{$n$}};
 \node at (1.5, -.7) {\scriptsize{$a \in (0, 1)$}};
 \node at (1.5, -1) {\scriptsize{$b \in (0, \frac12]$}};
 \node at (1.5, -1.3) {\scriptsize{$1-a \leq 2b$}};
\end{tikzpicture}
&
\begin{tikzpicture}[baseline = -.1cm]
\node at (0, .5){$L(\F_{t})$};
\node at (0, -.5){$t = (n-2)a^2 - 4b^{2} - 2ab + 2a + 3b$};
\end{tikzpicture}\\
\hline

\end{tabular}
\end{center}
\end{figure}

The utility of free transport, from the perspective of operator algebras, comes from its ability to establish embeddings and isomorphisms of C*-algebras and von Neumann algebras. It was first studied by Guionnet and Shlyakhtenko in \cite{MR3251831}, wherein they established a criterion for self-adjoint operators to generate C* and von Neumann algebras isomorphic to the C* and von Neumann algebras (respectively) generated by a family of free semicircular operators. In particular, the criterion is that the joint law of these self-adjoint operators (with respect to some tracial state) satisfy a formula called the \emph{Schwinger--Dyson equation}. The Schwinger--Dyson equation is actually a class of equations parameterized by non-commutative power series called \emph{potentials}. It is known that the joint law of free semicircular operators $x_1,\ldots, x_n$ satisfies the Schwinger--Dyson equation with the quadratic potential:
	\[
		V_0:=\frac12 \sum_{i=1}^n x_n^2.
	\]
Guionnet and Shlyakhtenko showed that if self-adjoint operators have a joint law satisfying the Schwinger--Dyson equation with a potential that is a sufficiently small perturbation of $V_0$, then they generate the same C* and von Neumann algebras as a free semicircular family. It is in this sense that we think of the free semicircle law as a \emph{distributional focal point}: self-adjoint operators with joint laws which are ``close'' to the free semicircle law generate the same C* and von Neumann algebras as a free semicircular family.

Several examples of self-adjoint operators whose joint law satisfy the aforementioned criterion have been demonstrated. Using estimates of Dabrowski from \cite{MR3270000}, Guionnet and Shlyakhtenko originally showed that the generators of the $q$-deformed free group factors satisfy a Schwinger--Dyson equation, and for sufficiently small parameter $|q|$ are isomorphic to the free group factors. Later, the second author and Zeng established a similar result for the generators of the mixed $q$-Gaussian algebras of \cite{MR1289833} in both the finite and infinite variable cases (see \cite{MR3531185} and \cite{NZ15}).

Interestingly, in the non-tracial setting the distributional focal point is no longer the free semicircle law, but instead the joint law of semicircular operators $x_1,\ldots, x_n$ generating a free Araki-Woods factor from \cite{MR1444786}. The corresponding potentials are quadratic potentials of the form:
	\[
		V_A:=\frac12 \sum_{j,k=1}^n \left[\frac{1+A}{2}\right]_{jk} x_kx_j,
	\]
where $A$ is self-adjoint matrix associated to $x_1,\ldots, x_n$ and the particular free Araki-Woods factor they generate. This was established by the second named author in \cite{MR3312436}, wherein it was also shown that the generators of the $q$-deformed free Araki-Woods algebras of \cite{MR1915438} satisfy a Schwinger--Dyson equation, and for sufficiently small parameter $|q|$ are isomorphic to free Araki-Woods factors.

In this paper, we consider an operator-valued setting, and show that the interpolated free group factors offer yet more distributional focal points. The corresponding potentials are:
	\[
		V_\mu:=\frac{1}{2} \sum_{\e\in\vec{E}} \mu(\e) x_\e^* x_\e,
	\]
where $(\Gamma,E,V,\mu)$ is a finite, connected, undirected, weighted graph so that $\cM(\Gamma,\mu)$ is isomorphic to an interpolated free group factor. We remark that non-perturbative operator valued transport has been considered in \cite{1701.00132}.

We note that different choices of weightings $\mu$ correspond to (potentially) different interpolated free group factors. Thus, it is tempting to ask if $V_\mu$ and $V_{\tilde{\mu}}$ can be ``close'' for different weightings $\mu$ and $\tilde{\mu}$, but the operator-valued setting will preclude such comparisons.


\subsection*{Acknowledgements} 

The authors would like to thank Dimitri Shlyakhtenko for many positive and helpful conversations.  Brent Nelson's work is supported by NSF grant DMS-1502822.

\section{Free graph algebra}

Suppose $(\Gamma, V, E, \mu)$ is a finite, connected, weighted, and undirected graph with vertex set $V$, a weighting function $\mu: V \rightarrow (0, \infty)$ satisfying $\sum_{v \in V} \mu(v) = 1$, and edge set $E$. We form the directed version of $\Gamma$,  $(\vec{\Gamma}, V, \vec{E}, \mu)$.  The edge set $\vec{E}$ of this directed graph is determined as follows:

\begin{itemize}

\item For each $e \in E$ having two distinct vertices $v$ and $w$ as endpoints,  there are two edges $\e$ and $\e^{\op}$ in $\vec{E}$.  We have $s(\e) = v$, $t(\e) = w$, $s(\e^{\op}) = w$ and $t(\e) = v$.

\item For each $e \in E$ serving as a loop at a vertex, $v$, there is one edge $\e \in \vec{E}$ which is a loop based on $v$.  For such loops, we will declare $\e = \e^{\op}$.

\end{itemize}
The mapping $\e \mapsto \e^{\op}$ induces an involution on $\vec{E}$. We let $\Pi$ and $\Lambda$ denote the set of paths and loops in $\vec{\Gamma}$, respectively. It will be convenient later to define for $\e\in \vec{E}$ the quantity $\mu(\e):=\sqrt{\mu(s(\e))\mu(t(\e))}$.

We denote by $\ell^{\infty}(V)$ the space of complex valued functions on $V$, and by $p_{v}$ the indicator function on $v \in V$. We explicitly construct the \emph{free graph algebra} as follows: Let $\C[\vec{E}]$ be the the complex vector space with basis $\vec{E}$.  $\C[\vec{E}]$ comes equipped with a $\ell^{\infty}(V)-\ell^{\infty}(V)$ bimodule structure determined by
$$
	p_{v}\cdot\e\cdot p_w = \delta_{v, s(\e)}\delta_{w,t(\e)} \e
$$
and $\ell^{\infty}(V)$-valued inner product given by
$$
\langle \e | \e' \rangle_{\ell^{\infty}(V)} = \delta_{\e, \e'} p_{t(\e)}
$$
which is extended to be linear in the right variable.

We now define the \emph{Fock space} of $\Gamma$, $\cF(\Gamma)$ to be the right C*-Hilbert module
$$
\cF(\Gamma) = \ell^{\infty}(V) \oplus \bigoplus_{n \geq 1} \C[\vec{E}]^{\otimes^{n}_{\ell^{\infty}(V)}}.
$$
$\cF(\Gamma)$ has a canonical left action by $\ell^{\infty}(V)$ given by bounded, adjointable operators: $p_{v}\cdot \e_{1}\otimes \cdots \otimes \e_{n} = \delta_{v, s(\e_{1})} \e_{1} \otimes \cdots \otimes \e_{n}$.  For each $\e \in \vec{E}$, we define the creation operator $\ell(\e)$ by
	\begin{align*}
		\ell(\e)& \cdot p_{v} = \delta_{v, t(\e)} \e \\
		\ell(\e)& \cdot \e_{1} \otimes\cdots\otimes \e_{n} = \e \otimes \e_{1} \otimes \cdots \otimes \e_{n}.
	\end{align*}
$\ell(\e)$ is bounded and adjointable with adjoint given by
	\begin{align*}
		\ell(\e)^{*}& \cdot p_{v} = 0\\
		\ell(\e)^{*}& \cdot \e_{1} \otimes \cdots \otimes \e_{n} = \langle \e|\e_{1}\rangle \e_{2} \otimes \cdots \otimes \e_{n}
	\end{align*}

For $\e \in \vec{E}$ we set
	\[
		x_{\e} = \sqrt[4]{\frac{\mu(s(\e))}{\mu(t(\e))}}\ell(\e) +  \sqrt[4]{\frac{\mu(t(\e))}{\mu(s(\e))}}\ell(\e^{\op})^{*}
	\]
Note that we have $p_{s(\e)}x_{\e}p_{t(\e)} = x_{\e}$, and $x_{\e}^{*} = x_{\e^{\op}}$. This implies that $x_{\e_1}\cdots x_{\e_n}=0$ unless $\e_1\e_2\cdots \e_n\in \Pi$.

We denote $\cS(\Gamma, \mu)$ to be the C*-algebra generated by $\ell^\infty(V)$ and $(x_{\e})_{\e \in \vec{E}}$. From \cite{MR1704661,1401.2485}, there is a faithful conditional expectation $E: \cS(\Gamma, \mu) \rightarrow \ell^{\infty}(V)$ given by
$$
E(x) = \sum_{v \in V} \langle p_{v} | xp_{v}\rangle_{\ell^{\infty}(V)}.
$$ 
Let $\tau: \cS(\Gamma, \mu) \rightarrow \C$ be given by  $\tau(x) = \mu \circ E(x)$. We call $\tau$ the \emph{free graph law corresponding to $(\Gamma,\mu)$}.  As shown in \cite{1401.2485}, $\tau$ is a faithful tracial state on $\cS(\Gamma, \mu)$.   We denote $\cM(\Gamma, \mu)$ by the von Neumann algebra generated by $\cS(\Gamma, \mu)$ in the GNS representation under $\tau$. Note that $\|x_\e\|^{2}_{2}=\mu(\e)$ for all $\e\in \vec{E}$.  We have the following theorems about the structures of $\cS(\Gamma, \mu)$ and $\cM(\Gamma, \mu)$.

\begin{thm*}[\cite{MR3110503}]
Suppose $(\Gamma,V,E,\mu)$ is a finite, connected, unoriented, weighted graph with at least two edges.  If $\alpha,\beta \in V$, then write $\alpha \sim \beta$ if $\alpha$ and $\beta$ are joined by at least one edge, and let $n_{\alpha, \beta}$ be the number of edges joining with $\alpha$ and $\beta$ as endpoints.  Finally, let $V_{>}$ be the set of vertices, $\beta$ satisfying $\mu(\beta) > \sum_{\alpha \sim \beta} n_{\alpha, \beta}\mu(\alpha)$.  We have
$$
\cM(\Gamma, \mu) \cong L(\F_{t}) \oplus \bigoplus_{\gamma \in V_{>}} \overset{r_{\gamma}}{\C}
$$
where $r_{\gamma} \leq p_{\gamma}$ and $\tau(r_{\gamma}) = \mu(\gamma) - \sum_{\alpha \sim \gamma} n_{\alpha, \beta}\mu(\alpha)$.  Moreover, the parameter, $t$, can be computed using Dykema's ``free dimension" formulas \cite{MR1201693, MR3164718}.  In particular, $\cM(\Gamma, \mu)$ is a factor if and only if $V_{>}$ is empty.
\end{thm*}  

We note that if $\cM(\Gamma, \mu)$ is a factor, necessarily $L(\F_{t})$, then Dykema's free dimension calculations give
\begin{equation}\label{eqn:parameter}
t = 1 - \sum_{v \in V} \mu(v)^{2} + \sum_{v \in V}\mu(v)\sum_{w \sim v} n_{v,w}\mu(w)
\end{equation}

\begin{thm*}[\cite{1509.0255}]
Let $\Gamma$ and $V_{>}$ be as in the statement of the previous theorem .  Let $V_{=}$ be the set of vertices $\beta$ satisfying $\mu(\beta) = \sum_{\alpha \sim \beta} n_{\alpha, \beta}\mu(\alpha)$, and let $V_{\geq} = V_{>} \cup V_{=}$.  Let $I$ be the norm-closed ideal generated by some $p_{\alpha}$ with $\alpha \in V \setminus V_{\geq}$. Then $I$ contains $\set{p_{\beta}}{\beta \in V \setminus V_{\geq}}$ and does not intersect $\set{p_{\gamma}}{\gamma \in V_{\geq}}$.  In addition, $I$ is generated by $\set{x_{\e}}{\e \in \vec{E}}$.  Furthermore, we have
\be  

\item $I$ is simple, has unique tracial state, and has stable rank 1.

\item $I$ is unital if and only if $V_{=}$ is empty.   If $V_{=}$ is empty, then
$$
\cS(\Gamma, \mu) = I \oplus \bigoplus_{\gamma \in V_{>}} \overset{r_{\gamma}}{\C}
$$
with $r_{\gamma} \leq p_{\gamma}$ and $\tau(r_{\gamma}) = \mu(\gamma) - \sum_{\alpha \sim \gamma} n_{\alpha, \beta}\mu(\alpha)$. If $V_{=}$ is not empty, then 
$$
\cS(\Gamma, \mu) = \mathcal{I} \oplus \bigoplus_{\gamma \in V_{>}} \overset{r_{\gamma}}{\C}
$$
where $\mathcal{I}$ is unital, and the strong operator closures of $I$ and $\mathcal{I}$ coincide in $L^{2}(\cS(\Gamma, \mu), \tau)$, and $\mathcal{I}/I \cong \bigoplus_{\beta \in V_{=}} \C$. 

\item $K_{0}(I) \cong \Z\set{[p_{\beta}]}{\beta \in V\setminus V_{\geq}} \text{ and } K_{1}(I) = \{0\}$ where the first group is the free abelian group on the classes of projections $[p_{\beta}]$.  Furthermore, $K_{0}(I)^{+} = \set{x \in K_{0}(I)}{\tau(x) > 0} \cup \{0\}$.
\ee

In particular, $\cS(\Gamma, \mu)$ is simple with unique tracial state if and only if $V_{\geq}$ is empty.
\end{thm*}

\begin{rem}
The free graph algebra can also be constructed as follows:  For each pair $e, e' \in E$, we define the map $\eta_{e, e'}: \ell^{\infty}(V) \rightarrow \ell^{\infty}(V)$ to be the linear extension of
	\[
		\eta_{e, e'}(p_{v}) = \begin{cases} \delta_{e,e'} \sqrt{\frac{\mu(v)}{\mu(w)}} p_w & \text{if }w\sim_e v\\ 0 & \text{otherwise}
\end{cases}
	\]
If we let $M_{E}(\ell^{\infty}(V))$ be the $|E| \times |E|$ matrices over $\ell^{\infty}(V)$, we see that the mapping $\eta:   \ell^{\infty}(V) \rightarrow M_{E}(\ell^{\infty}(V))$ given by $(\eta(x))_{e, e'} = \eta_{e, e'}(x)$ is completely positive.  The free graph algebra will be realized as the C*-algebra $\Phi(\ell^{\infty}(V), \eta)$ from \cite{MR1704661}.  This C*-algebra is generated by $\ell^{\infty}(V)$ as well as self-adjoint operators $(X_{e})_{e \in E}$ with a faithful conditional expectation  $\Psi: \Phi(\ell^{\infty}(V), \eta) \rightarrow \ell^{\infty}(V)$ given by $\Psi(X_{e}aX_{e'}) = \eta_{e, e'}(a)$ for $a \in \ell^{\infty}(V)$.  See \cite{MR1704661} for more details.
\end{rem}

\section{Free differential calculus}

In this section we introduce differential operators and establish some notation.

\subsection{The edge differentials, cyclic derivatives, and notation}\label{diffops_notation}

We fix a finite, unoriented, weighted graph $(\Gamma,V,E, \mu)$. Denote $\cM(\Gamma,\mu)$ by $\cM$. Let $A$ be the complex $*$-algebra generated by $\ell^{\infty}(V)$ and $(x_{\e})_{\e \in \vec{E}}$. From \cite{1509.0255} we have derivations $\p_\e\colon A\to A\otimes A^{\op}$  for each $\e\in \vec{E}$ given by:
	\[
		\p_{\e}(x_{\e'})=\delta_{\e,\e'} p_{s(\e)}\otimes p_{t(\e)}
	\]
and the Leibniz rule. These are known as \emph{free difference quotients}.

We have the following lemma from \cite{1509.0255}.

\begin{lem}\label{adjoint_formula}

For each $\e\in \vec{E}$, the derivation $\p_\e$ is closable as a densely defined operator from $L^2(\cM)$ to $L^2(\cM)\otimes L^2(\cM^{\op})$. Moreover, $A\otimes A^{\op}$ is the domain of $\p_\e^*$ and in particular
	\[
		\p_\e^*(p_{s(\e)}\otimes p_{t(\e)}) = \sqrt{\mu(t(\e))\cdot \mu(s(\e))} x_\e.
	\]

\end{lem}

As our free difference quotients are valued in $A\otimes A^{\op}$, we establish the following notation for this algebra:

\begin{itemize}

\item Let $m\colon A\otimes A^{\op}\to A$ be the linear extension of the multiplier map: $m(a\otimes b)=ab$.

\item Let $\sigma: A \otimes A^{\op} \rightarrow  A \otimes A^{\op}$ be the linear extension of the flip: $\sigma(a \otimes b) = b \otimes a$.

\item We define an adjoint on $A\otimes A^{\op}$ by $(a\otimes b)^*=a^*\otimes b^*$.

\item We also consider the conjugate linear involution on $A\otimes A^{\op}$ defined by $(a\otimes b)^\dagger = b^*\otimes a^*$. Note that $(a\otimes b)^\dagger = \sigma( (a\otimes b)^*)$.

\item Note that multiplication in $A\otimes A^{\op}$ is given by
	\[
		(a\otimes b)(c\otimes d) = (ac) \otimes (db),
	\]
for $a,b,c,d\in A$.

\item We let $\#$ denote the standard action of $A\otimes A^{\op}$ on $A$:  for $a \in A$ and $b\otimes c\in A \otimes A^{\op}$
	\[
		(b\otimes c)\# a := bac.
	\]
	
\end{itemize}

We will also consider a particular compression of  $M_{|\vec{E}|}(A\otimes A^{\op})$, the algebra of $|\vec{E}|$ by $|\vec{E}|$ matrices over $A\otimes A^{\op}$. We let $P$ be the diagonal matrix satisfying $[P]_{\e\e}=p_{s(\e)}\otimes p_{t(\e)}$, and we let $\bbM(A)$ be the compression $PM_{|\vec{E}|}(A\otimes A^{\op})P$. On this algebra, we establish the following notation:
	\begin{itemize}
	\item For $Q \in \bbM(A)$, define $Q^{T} \in \bbM(A)$ by $[Q^{T}]_{\e\e'} = \sigma([Q]_{(\e')^{\op}\e^{\op}})$.

	\item For $Q \in \bbM(A)$, define $Q^{*}\in \bbM(A)$ by $[Q^{*}]_{\e \e'} = [Q]_{\e'\e}^{*}$.

	\item  For $Q \in \bbM(A)$, define $Q^{\dagger}\in \bbM(A)$ by $[Q^{\dagger}]_{\e\e'} = [Q]_{\e^{\op}(\e')^{\op}}^\dagger$. Note that $Q^\dagger = (Q^*)^T$.
	
	\item For $Q_1, Q_2\in \bbM(A)$, define
		\[
			\langle Q_1,Q_2\rangle = (\tau\otimes\tau)\circ \text{Tr}( Q_1^* Q_2)
		\]
	\end{itemize}

We will also consider vectors of a particular form that respect the graph structure. $A^{\vec{E}}$ will denote the set of functions $f\colon \vec{E}\to A$ with the condition that
	\[
		f(\e)= (p_{s(\e)}\otimes p_{t(\e)}) \# f(\e) = p_{s(\e)} f(\e) p_{t(\e)} \qquad \forall \e\in \vec{E}.
	\]
These are $\vec{E}$-tuples such that the entry corresponding to $\e\in \vec{E}$ is a linear combination of paths in $\vec{\Gamma}$ that begin at $s(\e)$ and end at $t(\e)$. We will often write $f_\e$ to mean $f(\e)$. Write $x \in A^{\vec{E}}$ for the vector $(x)_{\e} = x_{\e}$. On this space, we establish the following notation:
	\begin{itemize}
	\item For $f\in A^{\vec{E}}$, we define $f^*\in A^{\vec{E}}$ by $(f^*)_\e = f_{\e^{\op}}^*$.

	\item For $f_1,f_2\in A^{\vec{E}}$, we define a dot product $f_1\# f_2\in A$ by $f_1\# f_2:=\sum_{\e} (f_1)_\e (f_2)_{\e^{\op}}$. 
	
	\item For $f_1, f_2\in A^{\vec{E}}$ we define an inner product on $A^{\vec{E}}$ defined by $\langle f_1, f_2\rangle := \tau(f_1^*\# f_2)$.
	
	\item We use $\#$ to also denote the standard action of $\bbM(A)$ on $A^{\vec{E}}$: for $Q\in\bbM(A)$ and $f\in A^{\vec{E}}$
	\[
		(Q\# f)_\e= \sum_{\e'\in \vec{E}} [Q]_{\e\e'}\# f_{\e'}.
	\]
Thus, $A^{\vec{E}}$ is characterized by $f=P\# f$ for all $f\in A^{\vec{E}}$.

	\item Observe that for $Q\in \bbM(A)$ and $f\in A^{\vec{E}}$, we have $(Q\# f)^* = Q^\dagger \# f^*$. Also, for $f_1,f_2\in A^{\vec{E}}$, we have $\tau((Q\# f_1) \# f_2 )= \tau(f_1 \# (Q^T\# f_2))$ and $\langle Q\# f_1, f_2\rangle = \langle f_1, Q^*\# f_2\rangle$.
	\end{itemize}

We let $\J\colon A^{\vec{E}}\to \bbM(A)$ denote the non-commutative Jacobian:
	\[
		[\J f]_{\e\phi} := \p_{\phi}(f_\e).
	\]
By our definition of $A^{\vec{E}}$, note that $\p_{\phi}(f_\e)= (p_{s(\e)}\otimes p_{t(\e)})\cdot\p_{\phi}(f_\e)\cdot (p_{s(\phi)}\otimes p_{t(\phi)})$, so in fact $\J f \in \bbM(A)$. In particular, observe that $\J x = P$. It follows from an easy computation that for $Q\in \bbM(A)$
	\begin{align}\label{eqn:J*_formula}
		\J^*(Q) = \left( \sum_{\phi\in \vec{E}} \partial_{\phi}^*([Q]_{\e \phi})\right)_{\e\in \vec{E}},
	\end{align}
when $\J\colon \overline{A^{\vec{E}}}^{\langle\cdot,\cdot\rangle}\to L^2(\bbM(A), (\tau\otimes \tau)\circ \text{Tr})$ is thought of as a densely defined operator.

Given $\e\in \vec{E}$, we set $\D_{\e}$ to be the \emph{cyclic derivative} $\D_\e=m\circ\sigma\circ \p_{\e^{\op}}$.  Note that
	\[
		\D_\e a= \sum_{a=a_1 x_{\e^{\op}} a_2} a_2a_1
	\]
whenever $a$ is a monomial in $(x_{\e})_{\e \in \vec{E}}$, and it is easy to check that $\D_\e(a)^* = \D_{\e^{op}}(a^*)$. Also note that $\D_\e a= (p_{s(\e)}\otimes p_{t(\e)}) \# \D_\e a$. In particular, we have $\D_\e (x_{\e_1}\cdots x_{\e_n})=0$ unless $\e_1\cdots\e_n\in \Lambda$. We then define $\D\colon A\to A^{\vec{E}}$ to be the \emph{cyclic gradient}:
	\[
		(\D a)_\e = \D_\e a\qquad \forall a\in A,\ \forall \e\in \vec{E}.
	\]

\subsection{The Banach algebra, $B_R$}

Suppose $(\Gamma, V, E, \mu)$ is a finite, weighted, undirected graph.  We denote by $B$ the universal unital $*$-algebra generated by  $\ell^\infty(V)$ and  $\{z_{\e}\, : \, \e \in \vec{E}\}$ subject to the following relations:

\begin{itemize}
\item $z_{\e}^{*} = z_{\e^{\op}}$

\item $p_{v}z_{\e}p_w = \delta_{v, s(\e)}\delta_{w, t(\e)}z_{\e}$.

\end{itemize}
Notice that these relations imply that $z_{\e_{1}}z_{\e_{2}}\cdots z_{\e_{n}}$ is nonzero if and only if $\e_{1}\cdots\e_{n}\in \Pi$.  There is a canonical $*$-homomorphism $B \rightarrow \cS(\Gamma, \mu)$ which is the identity on $\ell^\infty(V)$ and maps $z_{\e} \mapsto x_{\e}$ for all $\e \in \vec{E}$.

We also have $B^{\vec{E}}$ and $\bbM(B)$ defined in the same way as in Subsection \ref{diffops_notation} with the same conventions and notations. In particular, we have differential operators on $B$ which correspond via the $*$-homomorphism $B\to \cS(\Gamma, \mu)$ to the free difference quotients, cyclic derivatives, cyclic gradients, and non-commutative Jacobian from subsection \ref{diffops_notation}. We will denote these operators in the same way when the context is clear, and as $\{\partial_{z_\e}\}_{\e\in \vec{E}}, \{\D_{z_\e}\}_{\e\in\vec{E}}, \D_z,$ and  $\J_z$, respectively, otherwise.

Given $R > 0$, we place a norm on $B$ by
$$
\left\| \sum_{v \in V} b_{v}p_{v} +  \sum_{\e_{1}\dots\e_{n} \in \Pi} a_{\e_{1}\cdots\e_{n}}z_{\e_1}\dots z_{\e_n}  \right\|_{R} = \sup\{|b_{v}|\, | \, v \in V\} + \sum_{\e_{1}\dots\e_{n} \in \Pi} |a_{\e_{1}\cdots\e_{n}}|R^{n}. 
$$
We denote the completion of $B$ with respect to $\|\cdot\|_{R}$ by $B_R$. One can think of $B_R$ as power series in the $z_\e$ with radius of convergence at least $R$ and ``constant terms'' supported on $p_v$ for $v\in V$.  We have the following important fact.

\begin{fact}\label{contractive_extension}

 Suppose $\cA$ is any Banach $*$-algebra, and $\pi: B \rightarrow \cA$ is a unital $*$-homomorphism.  If the following two conditions hold
\begin{enumerate}

\item $\|\sum_{v \in V} a_{v}\pi(p_{v})\|_{\cA} \leq \sup\{|a_{v}| \, | \, v \in V\}$

\item $\|\pi(x_{\e})\| \leq R$ for all $\e \in \vec{E}$

\end{enumerate} 
Then $\pi$ extends to a contractive $*$-homomorphism from $B_R$ into $\cA$.  
\end{fact}

We note that when
	\[
		R\geq \max_{\e\in\vec{E}} \|x_\e\| = \max_{\e\in\vec{E}} \sqrt{2 + \sqrt{\frac{\mu(s(\e))}{\mu(t(\e))}} +  \sqrt{\frac{\mu(t(\e))}{\mu(s(\e))}}},
	\]
(by \cite{1401.2485}) these hypothesis are satisfied for the canonical map from $B$ to $A$. Moreover, if the above is a strict inequality, Lemma \ref{adjoint_formula} and an argument similar to Lemma 37 in \cite{MR3270000} tells us that this map is injective. In this case, we denote by $A_R$ the image of $B_R$ in $\cS(\Gamma,\mu)$, and we define the norm $\|\cdot\|_R$ on $A_R$ in the obvious way.

It follows that $B_R$ contains all power series that appear as elements in $B_{R'}$ for any $R' \geq R$.  It is straightforward to see that the canonical map $B_{R'} \rightarrow B_R$ is injective, so we will sometimes realize $B_{R'}$ as a dense $*$-subalgebra of $B_{R}$ in the norm $\|\cdot\|_{R}$.

Recall that $\mu(\e)=\sqrt{\mu(s(\e))\cdot \mu(t(\e))}$. We define $\cN_\mu$ on $B$ to be the weighted number operator:
	\[
		\cN_\mu(z_{\e_1}\cdots z_{\e_n}) = (\mu(\e_1)+\cdots +\mu(\e_n))z_{\e_1}\cdots z_{\e_n},
	\]
and $\cN_\mu(p_v)=0$ for all $v\in V$. For $R'>R$, it is easy to see that $\cN_\mu$ extends to a bounded map $\cN_\mu\colon B_{R'}\to B_R$. We define $\Sigma_\mu$ as:
	\[
		\Sigma_\mu(z_{\e_1}\cdots z_{\e_n}) = \frac{1}{\mu(\e_1)+\cdots +\mu(\e_n)}z_{\e_1}\cdots z_{\e_n}
	\]
and $\Sigma_\mu(p_v)=0$ for all $v\in V$, which we note is the inverse of $\cN_\mu$ restricted elements with no $\ell^\infty(V)$ terms. We note that $\Sigma_\mu$ extends to a bounded map on $B_R$.

$B_R^{\vec{E}}$ will denote the set of functions $f: \vec{E} \rightarrow B_R$ with the condition that $f= P\# f$. $B_R^{\vec{E}}$ is equipped with the norm
	\[
		\| f \|_R = \max_{\e \in \vec{E}}\|f_{\e}\|_R
	\]

We will let $B\hat\otimes_R B^{\op}$ be the projective tensor product of $B_R$ and $B_{R}^{\op}$ equipped with norm $\|\cdot\|_{R\otimes_\pi R}$. We use the same notations and conventions as on $B\otimes B^{\op}$. Note that the action of $B\otimes B^{\op}$ on $B$ extends to a bounded action of $B\hat\otimes_R B^{\op}$ on $B_R$ with
	\[
		\|(a\otimes b)\# c\|_R\leq \|a\otimes b\|_{R\otimes_\pi R} \|c\|_R.
	\]

We let $\bbM(B_R)$ denote the compression of $M_{|\vec{E}|\times |\vec{E}|}(B\hat\otimes_R B^{\op})$ by $P$. For $Q\in \bbM(B_R)$, we define
	\[
		\|Q\|_{R\otimes_\pi R}:= \max_{\e\in\vec{E}} \sum_{\e'\in\vec{E}} \|[Q]_{\e\e'}\|_{R\otimes_\pi R}.
	\]
Note that the action of $\bbM(B)$ on $B^{\vec{E}}$ extends to a bounded action of $\bbM(B_R)$ on $\B^{\vec{E}}_R$ with
	\[
		\|Q\# f\|_R \leq \|Q\|_{R\otimes_\pi R} \|f\|_R.
	\]
	
The following results were observed in \cite{MR3251831} (see also \cite[Lemmas 3.1, 3.4]{NZ15} and \cite[Lemma 2.5]{MR3312436})

\begin{lem}\label{diff_op_norms}
Let $R>S$. Then $\displaystyle \sum_{\e\in\vec{E}} \D_{\e}\Sigma_\mu, \sum_{\e\in\vec{E}} \D_{\e}$, and $\displaystyle \sum_{\e\in\vec{E}} \partial_\e$ extend to bounded maps
	\begin{align*}
		\left\|\sum_{\e\in\vec{E}} \D_\e\Sigma_\mu\colon B_R\to B_R\right\|&\leq \left(R\min_{\e\in\vec{E}} \mu(\e)\right)^{-1},\\
			\left\|\sum_{\e\in\vec{E}} \D_{\e}\colon B_R\to B_S\right\| &\leq C(R,S),\text{ and}\\
			\left\|\sum_{\e\in\vec{E}} \partial_\e\colon B_R\to B\hat{\otimes}_S B^{\op}\right\| &\leq C(R,S),
	\end{align*}
where $C(R,S)=(eS\log(R/S))^{-1}$. Consequently
	\begin{align*}
		\left\| \D\Sigma_\mu \colon B_R\to B_R^{\vec{E}}\right\| &\leq \left(R\min_{\e\in\vec{E}} \mu(\e)\right)^{-1},\\
		 \left\| \D\colon B_R\to B_S^{\vec{E}}\right\| &\leq C(R,S),\text{ and}\\
		 \left\| \J\colon B_R^{\vec{E}}\to \bbM(B_S)\right\| & \leq C(R,S).
	\end{align*}
\end{lem}

\subsection{The Schwinger--Dyson equation}

\begin{defn}\label{defn:S-D_equation}
For $R>0$ and $V\in B_R$, a linear functional $\varphi\colon B_R\to \C$ is said to \textbf{satisfy} (or \textbf{is a solution of}) \textbf{the Schwinger--Dyson equation with potential $V$} if for each $\e\in \vec{E}$ and $g\in B$
	\begin{equation}\label{S-D}
		\varphi( [\D_\e V]^* g) = \varphi\otimes\varphi^{\op}(\partial_\e g).
	\end{equation}
Equivalently, for each $f\in B^{\vec{E}}$
	\[
		\left\langle \D V, f\right\rangle = \left\langle P, \J f \right\rangle
	\]
Viewing $\J\colon \overline{B^{\vec{E}}}^{\langle\cdot,\cdot\rangle}\to L^2(\bbM(B), (\varphi\otimes \varphi)\circ \text{Tr})$ as a densely defined operator, this is further equivalent to saying $\J^*(P)= \D V$.

Let $\cA$ be a Banach $*$-algebra with equipped with a linear functional $\varphi\colon \cA\to\bbC$, and let $\pi\colon B\to \cA$ be a unital $*$-homomorphism satisfying the two conditions in Fact \ref{contractive_extension}. Let us still dneote by $\pi$ the contractive $*$-homomorphism $\pi\colon B_R\to\cA$. We will say that \textbf{$\varphi$ satisfies the Schwinger--Dyson equation with potential $V$} if $\varphi\circ \pi\colon B_R\to \bbC$ does.
\end{defn}

Let $M \in \bbM(B)$ be the diagonal matrix satisfying $[M]_{\e \e} = \mu(\e)[P]_{\e \e}$. Then by Lemma \ref{adjoint_formula},
	\[
		\J^{*}(P) = M \# x =  \D\left(\frac{1}{2} \sum_{\e \in \vec{E}} \mu(\e) x_{\e}^{*}x_{\e} \right).
	\]
Thus, for $R\geq \max \|x_\e\|$, $\tau$ satisfies the Schwinger--Dyson equation with potential
	\[
		V_\mu:=\frac{1}{2} \sum_{\e\in\vec{E}} \mu(\e) x_\e^* x_\e.
	\]

The Schwinger--Dyson equation with quadratic potential corresponding to a graph with one vertex was studied in \cite{MR3251831}. The joint law of a free semicircular family is the unique solution to this Schwinger--Dyson equation, and it was shown that small perturbations to this quadratic potential have solutions to the corresponding Schwinger--Dyson equation coming from the joint law of operators that generate the same C* and von Neumann algebras as a free semicircular family. Moreover, these operators can be realized as an invertible family of non-commutative power series in the free semicircular operators. Here, we study more general graphs and weighting, and note that the focal von Neumann algebras are generalized from free group factors to interpolated free group factors.

The existence and uniqueness of solutions to the Schwinger--Dyson equation is highly non-trivial. The following proposition says that so long as potentials are close to some $V_\mu$, then the Schwinger--Dyson equation has a unique solution. The majority of Section \ref{free_transport} is dedicated to showing the existence of solutions for potentials close to $V_\mu$. Aside from the inclusion of condition (ii) (which is innocuous but essential to this operator-valued setting) the proof is identical to that in \cite[Theorem 2.1]{MR2249657}.

\begin{prop}\label{Unique_SD}
Fix a weighting $\mu$ on the vertices $V$. Given $C>0$ and $R>C+2$, there exists a constant $K>0$ such that if $\|V-V_\mu\|_R<K$, then there is at most one linear functional $\varphi\colon B_R\to \C$ such that
	\begin{enumerate}
	\item[(i)] $\varphi$ satisfies the Schwinger--Dyson equation with potential $V$;
	
	\item[(ii)] $\varphi(p_v) = \mu(v)$ for all $v\in V$; and
	
	\item[(iii)] $|\varphi(z_{\e_1}\cdots z_{\e_d})|\leq C^d$ for all $\e_1,\ldots, \e_d\in \vec{E}$.
	\end{enumerate}
\end{prop}

\section{Free Transport}\label{free_transport}

We fix a finite, undirected, weighted graph $(\Gamma, V, E, \mu)$. Recall that the Schwinger--Dyson equation which $\tau$ satisfies is:
$$
\J^{*}(P) = M \# x =  \D V_\mu
$$
where $\mu(\e) = \sqrt{\mu(s(\e))\mu(t(\e))}$ and $V_{\mu} = \frac{1}{2} \sum_{\e \in \vec{E}} \mu(\e)x_{\e}^{*}x_{\e}$.

Throughout this section we will fix some $ R>\max_{\e\in\vec{E}} \|x_\e\|$ and $W\in B_R$. We are interested in finding a $y\in \cS(\Gamma,\mu)^{\vec{E}}$ whose joint law with respect to $\tau$ satisfies the Schwinger--Dyson equation with potential $V_\mu+W$. It will turn out that when $\|W\|_R$ is sufficiently small, there exists such a $y$  of the form $x+f$, for $f=\D g$ and $g\in A_R$ with $\|g\|_R$ small. Thus we will assume outright that $y=x+f$ and examine the implications of this equality on the Schwinger--Dyson equation. First note that the Schwinger--Dyson equation we are interested in solving is
	\[
		\J_y^{*}(P) = M \# y + (\D W)(y) = \D(V_{\mu} + W)(y).
	\]
We will express this entirely in terms of $x$  (given that $y=x+f$) using a change of variables formula. We will then make an effort to write both sides of the equation as cyclic gradients. This final form (\textit{cf.} Corollary \ref{final_form}) will be amenable to a fixed point argument (\textit{cf.} Theorem \ref{F_Lipschitz}).

Many of the results in this section follow \emph{mutatis mutandis} from proofs for the corresponding results in Section 3 of \cite{MR3251831}. For the reader's convenience, we have included proofs in the appendix.

\subsection{Equivalent forms of the Schwinger--Dyson equation}

When necessary we will add subscripts to the differential operators when we need to differentiate between $\partial_\e$ and $\partial_{y_\e}$, for example.

\begin{lem}\label{lem:change_of_var}
Let $y=x+f$ for $f\in A_R^{\vec{E}}$. Assume $\J y$ is invertible in $\bbM(A_R)$. For each $\e\in \vec{E}$, define
	\[
		\tilde{\p}_\e(q) = \sum_{\omega\in\vec{E}} \partial_{\omega}(q) \# [ (\J y)^{-1}]_{\omega \e},
	\]
for $q\in A_R$. Then:
\begin{enumerate}

\item[(i)] $\tilde{\p}_\e = \p_{y_\e}$ for all $\e\in\vec{E}$.

\item[(ii)] $\p_{y_{\e}}^{*}([P]_{\e\e}) = \sum_{\e' \in \vec{E}} \p_{\e'}^{*}([(\J y)^{-1}]_{\e' \e}^{*})$ i.e. $\J^{*}_{y}(P) = \J^{*}([(\J y)^{-1}]^{*})$

\item[(iii)] If $y = \D g$ with $g \in A_R$, then $(\J y)^{T} = \J y$.  Furthermore if $g = g^{*}$, then $(\J y)^{*} = \J y$.

\end{enumerate}
\end{lem}
\begin{proof}
For $\e,\phi\in \vec{E}$ we have
	\begin{align*}
		\tilde{\p}_\e(y_{\phi}) &= \sum_{\omega\in\vec{E}} \partial_{\omega}(y_\phi) \#[ (\J y)^{-1}]_{\omega\e}= \sum_{\omega\in\vec{E}} [\J y]_{\phi\omega} \#[ (\J y)^{-1}]_{\omega\e}= [P]_{\phi\e}= \p_{y_\e}(y_\phi).
	\end{align*}
This implies (i). Next we compute for $q\in A$:
	\begin{align*}
		\left\langle\p_{y_\e}^*([P]_{\e\e}), q\right\rangle &= \left\langle [P]_{\e\e} , \sum_{\omega\in \vec{E}} \p_{\omega}(q)\# [(\J y)^{-1}]_{\omega\e}\right\rangle\\
			&= \sum_{\omega\in\vec{E}} \left\langle [(\J y)^{-1}]_{\omega\e}^*, \p_{\omega}(q)\right\rangle\\
			&= \left\langle \sum_{\omega\in\vec{E}} \p_{\omega}^*\left([(\J y)^{-1}]_{\omega\e}^*\right), q\right\rangle.
	\end{align*}
Thus $\p_{y_\e}^*([P]_{\e\e}) = \sum_{\omega\in\vec{E}} \p_{\omega}^*\left([(\J y)^{-1}]_{\omega\e}^*\right)$. The rest of (ii) follows from (\ref{eqn:J*_formula}).

To show (iii), it suffices to assume $g$ is a monomial. In this case
	\[
		[\J y]_{\e\phi} = \p_\phi(y_\e) = \p_\phi  \D_{\e} g =\sum_{g=ax_\phi b x_{\e^{\op}} c} ca\otimes b + \sum_{g=ax_{\e^{\op}}bx_\phi c} b\otimes ca.
	\]
Thus we have
	\[
		\sigma([\J y]_{\e\phi}) = \sum_{g=ax_\phi b x_{\e^{\op}} c} b\otimes ca + \sum_{g=ax_{\e^{\op}}bx_\phi c} ca\otimes b = [\J y]_{\phi^{\op}\e^{\op}}.
	\]
Consequently,
	\[
		[\J y^T]_{\e\phi} = \sigma([\J y]_{\phi^{\op}\e^{\op}})=[\J y]_{\e\phi},
	\]
so that $\J y^T =\J y$. This same computation implies
	\begin{align*}
		(\partial_\phi\D_e g)^\dagger &= \sum_{g=ax_\phi b x_{\e^{\op}} c} b^*\otimes a^*c^* + \sum_{g=ax_{\e^{\op}}bx_\phi c} a^*c^*\otimes b^*\\
			&= \sum_{g^*=c^*x_\e b^* x_{\phi^{\op}} a^*}  b^*\otimes a^*c^* + \sum_{g^*=c^*x_{\phi^{\op}}b^* x_\e a^*} a^*c^*\otimes b^* = \partial_{\phi^{\op}}\D_{\e^{\op}} g^*
	\end{align*}
Thus if $g=g^*$, then
	\[
		[\J y]_{\phi \e}^\dagger = [\J y]_{\phi^{\op} \e^{\op}}.
	\]
That is, $\J y^\dagger = \J y$. Consequently $(\J y)^* = (\J y^T)^\dagger = \J y$.
\end{proof}

\begin{prop}\label{prop_3.2}
Assume $y=x+f$ with $f=\D g$ and $g=g^*\in A_R$ for some $\displaystyle R>\max_{\e\in\vec{E}} \|x_\e\|$. Further assume $\J y$ is invertible in $\bbM(A_R)$. Then the Schwinger--Dyson equation with potential $V_\mu+W$ is equivalent to
		\begin{align*}
			-\J^{*}(\J f) - M \# f &= \D(W(X+f)) + (\J f) \# (M \# f)\\
							&+ (\J f) \# \J^{*}\left(\frac{\J f}{P + \J f}\right) -  \J^{*}\left(\frac{\J f^{2}}{P + \J f}\right)
		\end{align*} 
\end{prop}
\begin{proof}
This is the analogue of \cite[Lemma 3.3]{MR3251831}. A detailed \hyperref[proof_prop_3.2]{proof} can be found in the appendix.
\end{proof}

\begin{thm}\label{thm_3.3}
Let $f=\D g$, with $g=g^*\in A_R$ for some $R> \max_{\e\in\vec{E}} \|x_\e||$. For each $m\in \bbN$, the following equality holds:
\begin{align*}
\frac{1}{m}\D&\left[ (1 \otimes \tau + \tau \otimes 1)\circ\Tr\left( \J f^{m} \right) \right]\\
&= (\J f) \# \J^{*}(\J f^{m-1}) - \J^{*}(\J f^{m}) 
\end{align*}
\end{thm}
\begin{proof}
This is the analogue of \cite[Lemma 3.4]{MR3251831}. A detailed \hyperref[proof_thm_3.3]{proof} can be found in the appendix
\end{proof}

\begin{lem}
For $g\in A_R$, $\D(\cN_\mu g) = \cN_\mu \D g + M\# \D g$.
\end{lem}
\begin{proof}
It suffices to prove the result when $g = x_{\e_{1}}\cdots x_{\e_{n}}$, in which case we have
\begin{align*}
[\D(\cN_{\mu} g)]_{\e} &= \sum_{\e_{j} = \e^{\op}} [\mu(\e_{1}) + \cdots +\mu(\e^{\op}) + \cdots + \mu(\e_{n})]x_{\e_{j+1}}\cdots x_{\e_{n}}x_{\e_{1}}\cdots x_{\e_{j-1}}\\
&= \sum_{\e_{j} = \e^{\op}} [\mu(\e_{1}) + \cdots +\mu(\e) + \cdots + \mu(\e_{n})]x_{\e_{j+1}}\cdots x_{\e_{n}}x_{\e_{1}}\cdots x_{\e_{j-1}}\\
&= \sum_{\e_{j} = \e^{\op}} [(\mu(\e_{1}) + \cdots +\widehat{\mu(\e)} + \cdots + \mu(\e_{n})) + \mu(\e)]x_{\e_{j+1}}\cdots x_{\e_{n}}x_{\e_{1}}\cdots x_{\e_{j-1}}\\
&= [\cN_{\mu}\D g]_{\e} + [M \# \D g]_{\e}
\end{align*}
\end{proof}

\begin{lem}\label{lem:LHS_finalform}
For $f=\D g$, with $g=g^*\in A_R$ for some $R> \max_{\e\in\vec{E}} \|x_\e||$ the following equality holds:
$$-\J^{*}(\J f) - M \# f = \D\left[ (1 \otimes \tau + \tau \otimes 1)\circ\Tr\left( \J f \right) - \cN_{\mu}g \right]$$
\end{lem}
\begin{proof}
Note that if $m = 1$ in Theorem \ref{thm_3.3} we have
\begin{align*}
\D&\left[ (1 \otimes \tau + \tau \otimes 1)\circ\Tr\left( \J f \right) - \cN_{\mu}g \right]\\
&= (\J f) \# \J^{*}(P) - \D\cN_{\mu}(g) - \J^{*}((\J f))\\
&= (\J f) \# (M \# X) - \D\cN_{\mu}(g) - \J^{*}((\J f))\\
\end{align*} 
Note that 
$$
[(\J f) \# (M \# X)]_{\e} = \sum_{\phi\in\vec{E}} \p_{\phi}f_{\e} \# \mu(\phi)x_{\phi} = [\cN_{\mu}f]_{\e}
$$
so by the previous lemma we have the desired equality.
\end{proof}

\begin{prop}\label{prop:Jf_on_f}
Let $f=\D g$, with $g=g^*\in A_R$ for some $R> \max_{\e\in\vec{E}} \|x_\e||$. Then
	\[
		(\J f) \# (M \# f) = \frac{1}{2}\D(f \# (M \# f)).
	\]

\end{prop}

\begin{proof}
We have
\begin{align*}
\frac{1}{2}\D_{\e}(f \# (M \# f)) &= \frac{1}{2} \sum_{\phi} \D_{\e}(f_{\phi} \cdot \mu(\phi^{\op})f_{\phi^{\op}})\\
\text{[Using $\mu(\phi^{\op}) = \mu(\phi)$] }\qquad&= \frac{1}{2} \sum_{\phi} \sigma(\p_{\e^{\op}}(f_{\phi})) \# \mu(\phi^{\op})f_{\phi^{\op}} + \sigma(\mu(\phi)\p_{\e^{\op}}(f_{\phi^{\op}})) \# f_{\phi}  \\
\text{[Using $(\J f)^{T} = (\J f)$] }\qquad&= \frac{1}{2} \sum_{\phi} \p_{\phi^{\op}}(f_{\e}) \# \mu(\phi^{\op})f_{\phi^{\op}} + \p_{\phi}(f_{\e}) \# \mu(\phi)f_{\phi} \\
&= \frac{1}{2} \sum_{\phi} [\J f]_{\e\phi^{\op}} \# (M \# f)_{\phi^{\op}} + [\J f]_{\e\phi} \# (M \# f)_{\phi}\\
&= (\J f \# (M \# f))_{\e}
\end{align*}
as desired.
\end{proof}

\begin{cor}\label{final_form}
Assume $y=x+f$ with $f=\D g$ and $g=g^*\in A_R$ for some $\displaystyle R>\max_{\e\in\vec{E}} \|x_\e\|$. Further assume $\|\J f\|_{R\otimes_\pi R}<1$. Then the joint law of $y$ with respect to $\tau$ satisfies the Schwinger--Dyson equation with potential $V_\mu+W$ if and only if
\begin{align*}
	\D \cN_\mu g = \D\left[ - W(x+\D g) - \frac12 \D g\# (M\#\D g) -\sum_{m=1}^\infty \frac{(-1)^m}{m} (1\otimes \tau+\tau\otimes 1)\circ\Tr(\J\D g^m)\right]
\end{align*}
\end{cor}
\begin{proof}
Observe that the condition on $\J f$ implies $\J y = P+\J f$ is invertible in $\bbM(A_R)$. So, by Proposition \ref{prop_3.2}, Lemma \ref{lem:LHS_finalform}, and Lemma \ref{prop:Jf_on_f}, the Schwinger--Dyson equation with potential $V_\mu+W$ is equivalent to
	\begin{align*}
		\D\left[(1\otimes \tau + \tau\otimes 1)\circ\Tr(\J f) - \cN_\mu g\right] =& \D(W(x+f)) +\frac12 \D(f\#(M\# f))\\
			& + (\J f)\# \J^*\left(\frac{\J f}{P+\J f}\right)  - \J^*\left(\frac{\J f^2}{P+\J f}\right)
	\end{align*}
Moving terms we have
	\begin{align*}
		\D\cN_\mu g =&  \D\left[ -W(x+f)) - \frac12 f\# (M\# f) + (1\otimes \tau + \tau\otimes 1)\circ\Tr(\J f)\right]\\
			& - (\J f)\# \J^*\left(\frac{\J f}{P+\J f}\right)  + \J^*\left(\frac{\J f^2}{P+\J f}\right)
	\end{align*}
Using the expansion $\frac{P}{P+\J f} =P+\sum_{m=1}^\infty (-1)^m \J f^m$ and Theorem \ref{thm_3.3} this becomes
	\begin{align*}
		\D\cN_\mu g =&  \D\left[ -W(x+f)) - \frac12 f\# (M\# f) + (1\otimes \tau + \tau\otimes 1)\circ\Tr(\J f)\right]\\
			& - \sum_{m=1}^\infty \frac{(-1)^m}{m}\D[(1\otimes \tau+\tau\otimes 1)\circ\Tr(\J f^m)]
	\end{align*}
Substituting $f=\D g$ completes the proof.
\end{proof}

\subsection{Estimates on terms in the Schwinger--Dyson Equation}

Denote
	\[
		C=\max_{\e\in\vec{E}} \|x_\e\| 
	\]
Recall that if $\tau$ is the free graph law corresponding to $(\Gamma,V,E,\mu)$, then $\tau(x_{\e_1}\cdots x_{\e_n}) \leq C^{n}$ for all $n\in\N$ and $\e_1,\ldots, \e_n\in \vec{E}$. For a path $\gamma=\e_1\e_2\cdots \e_n$, we will write $x_\gamma$ for a $x_{\e_1}\cdots x_{\e_n}$ and $\mu(\gamma)$ for $\mu(\e_1)+\cdots +\mu(\e_n)$.

\begin{lem}\label{lem_3.8}
Let $R>2C$. For $m\geq 1$, define for $g_1,\ldots, g_m\in B_R$
	\[
		Q_m(g_1,\ldots, g_m) = (1\otimes \tau +\tau\otimes 1)\circ \Tr(\J\D g_1\cdots \J\D g_m)
	\]
Then
	\[
		\|Q_m(\Sigma_\mu g_1,\ldots, \Sigma_\mu g_m) \|_R \leq 2\left(\frac{2}{R^2\min \mu(\e)}\right)^m \prod_{i=1}^m \|g_i\|_R
	\]
\end{lem}
\begin{proof}
This is the analogue of \cite[Lemma 3.8]{MR3251831}. A detailed \hyperref[proof_lem_3.8]{proof} can be found in the appendix.
\end{proof}

\begin{lem}\label{lem_3.9}
Let $R>2C$. For $m\geq 1$, define $Q_m(g)=Q_m(g,\ldots, g)$ for $g\in B_R$. Then for $g,f\in B_R$ we have
	\[
		\| Q_m(\Sigma_\mu g) - Q_m(\Sigma_\mu f)\|_R \leq \sum_{k=0}^{m} 2\left(\frac{2}{R^{2}\min\mu(\e)}\right)^{m} \|g\|_{R}^{k-1}\|g-f\|_{R}\|f\|_{R}^{m-k}  
	\]
\end{lem}
\begin{proof}
This is the analogue of \cite[Lemma 3.9]{MR3251831}. A detailed \hyperref[proof_lem_3.9]{proof} can be found in the appendix.
\end{proof}

Note that by removing $\D$ from the Schwinger--Dyson equation as in Corollary \ref{final_form}, we have
$$
\cN_{\mu}g = -W(X + \D g) - \frac{1}{2}\D g \# (M \# \D g) - \sum_{m=1}^{\infty} \frac{(-1)^{m}}{m}Q_{m}(g)
$$
By making the substitution $g = \Sigma_{\mu} g'$, we obtain the equation
$$
g = -W(X + \D\Sigma_{\mu}g) - \frac{1}{2}\D\Sigma_{\mu} g \# (M \# \D\Sigma_{\mu} g) + \sum_{m=1}^{\infty} \frac{(-1)^{m+1}}{m}Q_{m}(\Sigma_{\mu}g)
$$

\begin{lem}\label{lem_3.10}
Let $R>2C$. For $g\in B_R$ define
	\[
		Q(g) = \sum_{m=1}^{\infty} \frac{(-1)^{m+1}}{m}Q_{m}(g).
	\]  
Then on $\set{g\in B_R}{ \|g\|_R < \displaystyle \frac12 R^{2}\min_{\e \in \vec{E}}\mu(\e)}$, $Q\circ\Sigma_\mu$ converges, is locally Lipschitz:
\begin{align*}
\|Q(\Sigma_{\mu}f) &- Q(\Sigma_{\mu}(g)\|_{R}\\
&\leq \|f - g\|_{R} \left(\frac{4}{R^{2}\min\mu(\e)}\right)\left( \frac{1}{\left(1 - \frac{2\|f\|_{R}}{R^{2}\min\mu(\e)}\right)\left(1 - \frac{2\|g\|_{R}}{R^{2}\min\mu(\e)}\right)} \right),
\end{align*}
 and is locally bounded:
$$
\|Q(\Sigma_{\mu}g)\|_{R} \leq \|g\|_{R} \left(\frac{4}{R^{2}\min\mu(\e)}\right)\left( \frac{1}{1 - \frac{2\|g\|_{R}}{R^{2}\min\mu(\e)}}\right).
$$
\end{lem}
\begin{proof}
This is the analogue of \cite[Lemma 3.11]{MR3251831}. A detailed \hyperref[proof_lem_3.10]{proof} can be found in the appendix.
\end{proof}

\begin{prop}\label{prop_3.11}
If $x \in B_{R}^{\vec{E}}$, and $y, z \in B_{R'}^{\vec{E}}$, then 
$$
\| W(x + y) -  W(x + z)\|_{R} \leq \sum_{\e \in \vec{E}} \|\p_{\e}W\|_{(R + R')\otimes_\pi (R+R')} \cdot \|y-z\|_{R}
$$
\end{prop}
\begin{proof}
Write $W(x) = \sum_{\gamma} c_{\gamma}x_{\gamma}$ where the sum is over paths $\gamma$ in $\vec{\Gamma}$.  Then 
	\begin{align*}
		W&(x + y) -  W(x + z)\\
			&= \sum_{n=0}^{\infty}\sum_{\e_{1}\dots\e_{n}}\sum_{j}c_{\e_{1}\dots\e_{n}}(x+z)_{\e_{1}}\dots (x+z)_{\e_{j-1}}(y_{\e_{j}}-z_{\e_{j}})(x+y)_{\e_{j+1}}\cdots (x+y)_{\e_{n}}
	\end{align*}
Since $\|x + z\| \leq R + R'$ and $\|x + y\| \leq R + R'$ it follows that 
\begin{align*}
\| W(x + y) -  W(x + z)\|_{R} &\leq  \sum_{\e} \|\p_{\e}W\|_{(R + R')\otimes_\pi (R+R')}\|y-z\|_{R}. \qedhere
\end{align*}
\end{proof}

\begin{thm}\label{thm:lip}
Let $R>2C$. Define $F$ on $B_{R}$ by 
$$
F(g) = -W(x + \D\Sigma_{\mu}g) - \frac{1}{2}\D\Sigma_{\mu} g \# (M \# \D\Sigma_{\mu} g) + Q(\Sigma_{\mu} g).
$$
Assume that $W \in B_{S}$ for $S > R+\frac1R$.
Then $F$ is locally bounded and locally Lipschitz on 
	\[
		\cB:=\set{g}{\|g\|_{R} \leq \min_{\e \in \vec{E}}\mu(\e)}
	\]
with
	\[
		\| F(g)\|_R \leq \|W\|_S+ \frac{\|g\|_R^2 \max \mu(\e)}{2 R^2 \min \mu(\e)^2}  + \frac{4\|g\|_R}{R^2\min\mu(\e)} \left(\frac{1}{1- \frac{ 2\|g\|_R}{ R^2 \min\mu(\e)}}\right)
	\]
and
	\begin{align*}
		\| F(g) - F(f)\|_R\\
			 \leq \|g-f\|_R &\left( \frac{1}{R\min \mu(\e)} \sum_{\e\in\vec{E}} \|\partial_{\e} W\|_{\frac{3}{2} R\otimes_\pi \frac32 R}+ \frac{\max\mu(\e)}{2 R^2 \min \mu(\e)^2} (\|g\|_R + \|f\|_R)\right.\\
			&\left. + \frac{4}{R^2\min\mu(\e)} \left(\frac{1}{1- \frac{2\|g\|_R}{R^2\min\mu(\e)}}\right)\left(\frac{1}{1- \frac{2\|f\|_R}{R^2\min\mu(\e)}}\right)\right)
	\end{align*}
\end{thm}
\begin{proof}
Note that for $g\in \cB$, $\|\D \Sigma_{\mu} g\|_{R} \leq \frac{1}{ R \min \mu(\e)}\|g\|_{R}\leq \frac1R$.  It follows that $W(X + \D\Sigma_{\mu}g)$ is well defined for $g\in \cB$. Also, $R>2C$ implies $R>2\sqrt{3}$. Hence $\|g\|_R < \frac12 R^2 \min\mu(E)$ for $g\in\cB$.  From Proposition \ref{prop_3.11}, we have
\begin{align*}
\|W(x + \D\Sigma_{\mu}g) - W(x + \D\Sigma_{\mu}f)\|_{R} &= \sum_{\e} \|\p_{\e}W\|_{\frac{3}{2}R\otimes_\pi \frac{3}{2}R}\|\D\Sigma_{\mu}g-\D\Sigma_{\mu}f\|_{R}\\
&\leq \frac{1}{R \min \mu(\e)} \sum_{\e} \|\p_{\e}W\|_{\frac{3}{2}R\otimes_\pi \frac{3}{2}R}\|g - f\|_{R}
\end{align*}
Next, since $M$ is a diagonal matrix and $\mu(\e)=\mu(\e^{\op})$ we have
\begin{align*}
\left\|\frac{1}{2}\D\Sigma_{\mu} g \right. & \# (M \# \D\Sigma_{\mu} g) - \left.\frac{1}{2}\D\Sigma_{\mu} f \# (M \# \D\Sigma_{\mu} f)\right\|_{R} \\
	\leq& \frac{1}{2}\|(M^{\frac12}\#\D\Sigma_{\mu} g )\# (M^{\frac12}\#\D\Sigma_{\mu}g) - (M^{\frac12}\#\D\Sigma_{\mu} f) \# (M^{\frac12}\#\D\Sigma_{\mu}f) \|_{R}\\
\leq& \frac{1}{2}\|M^{\frac12}\#\D\Sigma_{\mu} g\|_{R}  \|M^{\frac12}\#\D\Sigma_{\mu}(g - f)\|_{R}\\
&+   \frac{1}{2}\|M^{\frac12}\#\D\Sigma_{\mu} (g - f)\|_{R} \# \|M^{\frac12}\#\D\Sigma_{\mu}f\|_{R}\\
 \leq & \frac{\max\mu(\e)}{2R^2 \min \mu(\e)^{2}}(\|f\|_{R} + \|g\|_{R})\|f - g\|_{R}
\end{align*}
Putting these together along with Lemma \ref{lem_3.10}, we obtain the claimed local bound and local Lipschitz constant.
\end{proof}

\subsection{Existence of a solution}

\begin{assumptions}\label{ass:bounds} 
We now make the following assumptions:
\begin{enumerate}

\item We choose $R$ large enough so that $\displaystyle R\min_{\e\in\vec{E}}\mu(\e)>4$

\item We choose $S>R+\frac1R$.

\item Assume $W \in B_{S}$ with
	\begin{itemize}
	\item $\displaystyle\|W\|_S \leq \frac{1}{2}\min_{\e\in\vec{E}} \mu(\e)$
	\item $\displaystyle\|W\|_S \leq 2 e \left(R+\frac1R\right)\log\left(\frac{S}{R+\frac1R}\right)$.
	\end{itemize}
\end{enumerate}
\end{assumptions}

\begin{thm}\label{F_Lipschitz}
Let $R$ and $S$ be as in Assumptions \ref{ass:bounds}, and $\displaystyle\cB = \set{g}{\|g\|_{R} \leq \min_{\e\in\vec{E}}\mu(\e)}$.  Define $F$ as above by
$$
F(g) = -W(X + \D\Sigma_{\mu}g) - \frac{1}{2}\D\Sigma_{\mu} g \# (M \# \D\Sigma_{\mu} g) + Q(\Sigma_{\mu} g).
$$
then $F$ maps $\cB$ into itself and is a strict contraction on $\cB$.  Consequently, there exists a unique $g \in \cB$ with $F(g) = g$. Moreover, $\|g\|_R\leq 10 \|W\|_R$.
\end{thm}
\begin{proof}
We first note that the condition on $R$ implies
	\[
		2C\leq 2\sqrt{3+ \frac{1}{\sqrt{\min\mu(\e)}}} \leq \frac{4}{\min\mu(\e)}< R,
	\]
hence $R> 2C$. Moreover, this condition implies $R> 4$ and so for $g\in \cB$ we have
	\[
		\frac{2\|g\|_R}{R^2 \min\mu(\e)} \leq \frac{2}{ R^2} < \frac18.
	\]
Thus, the hypothesis of Theorem \ref{thm:lip} are satisfied, and this gives
\begin{align*}
\| F(g)\|_R& \leq \|W\|_S+ \frac{\|g\|_R^2 \max \mu(\e)}{2 R^2 \min \mu(\e)^2}  + \frac{4\|g\|_R}{R^2\min\mu(\e)} \left(\frac{1}{1- \frac{ 2\|g\|_R}{ R^2 \min\mu(\e)}}\right)\\
&< \frac{1}{2}\min \mu(\e)  + \|g\|_R \left( \frac{1}{32} + \frac{1}{4} \left( \frac{1}{1 - \frac18}\right)\right)\\
&\leq \frac12 \min \mu(\e) + \min\mu(\e)\left(\frac{1}{32} + \frac{2}{7}\right)< \min\mu(\e).
\end{align*}
So $F$ maps $\cB$ to itself. Now, by Lemma \ref{diff_op_norms} and our other assumption on $\|W\|_S$, we know
	\[
		\sum_{\e\in\vec{E}} \|\p_\e W\|_{(R+\frac{1}{R})\otimes_\pi (R+\frac1R)} \leq \left(e \left(R+\frac1R\right)\log\left(\frac{S}{R+\frac1R}\right)\right)^{-1} \|W\|_S \leq 2.
	\]
With this, the Lipschitz constant for $F$ from Theorem \ref{thm:lip} for $f,g\in\cB$ is bounded by
\begin{align*}
 \frac{2}{R\min\mu(\e)} + \frac{\max_{\e}\mu(\e)}{R^2\min \mu(\e)} +\left(\frac{4}{R^{2}\min\mu(\e)}\right)\left( \frac{1}{1 - \frac18}\right)^2  \leq\frac12+ \frac{1}{16} + \frac14 \left(\frac87\right)^2<\frac{9}{10}.
\end{align*}
Thus, $F$ is a strict contraction on $\cB$.

Now, let $g_0=0\in \cB$. Inductively define $g_n\in \cB$ by $g_n:=F(g_{n-1})$. Note that $g_1=F(0)=-W$. Then $g_n$ converges to the unique fixed point $g\in\cB$: $F(g)=g$. Note that
	\begin{align*}
		\|g_{n+1} - g_n\|_R = \| F(g_n) - F( g_{n-1})\|_R < \frac{9}{10} \|g_n - g_{n-1}\|_R < \cdots < \left(\frac{9}{10}\right)^n \|W\|_R.
	\end{align*}
Thus
	\[
		\|g_n\|_R \leq \sum_{k=0}^{n-1} \| g_{k+1} - g_k\|_R <\sum_{k=0}^{n-1} \left(\frac{9}{10}\right)^k \|W\|_R < 10\|W\|_R.
	\]
Consequently, $\|g\|_R\leq 10 \|W\|_R$.
\end{proof}

\begin{cor}\label{Exists_SD}
Assume $R$, $S$, and $W$ satisfy Assumptions \ref{ass:bounds}. Then there exists a constant $K>0$ depending on $R$, $S$, and $\mu$, so that if $\|W\|_S<K$, then there exists $y = x+\D g$ with $g \in A_R$ such that
$$
\J_y^{*}(P) = M \# y + (\D W)(y) = \D(V_{\mu} + W)(y).
$$
That is, the joint law of $y$ with respect to $\tau$ satisfies the Schwinger--Dyson equation with potential $V_\mu+W$.
\end{cor}
\begin{proof}
Let $\hat g\in \cB$ be the fixed point from Theorem \ref{F_Lipschitz}. Let $g\in A_R$ be the image of $\Sigma_\mu \hat g\in B_R$ under the canonical homomorphism. Fix some $R'\in [C, R)$ and observe that
	\[
		\| \J \D g\|_{R'\otimes_\pi R'} = \|\J\D \Sigma_\mu \hat g\|_{R'\otimes_\pi R'} \leq \frac{C(R, R')}{R\min\mu(\e)} \|\hat g\|_R \leq \frac{C(R, R')}{R\min\mu(\e)} 10\|W\|_R.
	\]
So, given $K$ sufficiently small, we can ensure $\| \J\D g\|_{R'\otimes_\pi R'}<1$. Consequently, $\J y = P+ \J\D g$ is invertible in $\bbM(A_{R'})$. Hence the Schwinger--Dyson equation is equivalent to equation in Corollary \ref{final_form}. This is satisfied by $g$ since $F(\hat g)= \hat g$.
\end{proof}

\section{Isomorphism results}

Let $R$ and $S$ be as in Assumptions \ref{ass:bounds}. In this section, we present our main result. We first observe the following inverse function theorem, which follows by the same argument as in \cite[Corollary 2.4]{MR3251831} or \cite[Lemma 3.6]{NZ15}.

\begin{prop}\label{inverse_function}
Let $R>R'>C$. Then there exists a constant $K\in (0, R-R')$ so that if $f\in A_R^{\vec{E}}$ with $\|f\|_R< K$, then there exists $h\in B_{R'+K}^{\vec{E}}$ satisfying $h(x+f) = x$.
\end{prop}

\begin{thm}\label{thm:iso}
Let $N$ be a C*-algebra equipped with a linear functional $\varphi$.  Suppose there exists a unital contractive $*$-homomorphism $\pi\colon B\to N$ with dense range, with $R> \max \|\pi(z_\e)\|$. Then there exists a constant $K>0$ depending only on $R$, $S$, and $\mu$ so that if $\varphi$ satisfies the Schwinger--Dyson equation with some potential $V\in B_S$ and $\|V-V_\mu\|_S<K$, then we have the following trace-preserving isomorphisms:
	\[
		N\cong \cS(\Gamma,\mu)\qquad\text{ and }\qquad W^*(N,\varphi) \cong \cM.
	\]
In particular, $\varphi$ is a faithful tracial state on $N$.
\end{thm}

\begin{proof}
Denote $W:=V-V_\mu$. The desired constant $K$ will simply be a minimum of the constants required to apply several of the previous results in this paper. Let $K_1$ be as in Proposition \ref{Unique_SD} and let $K_2$ be as in Corollary \ref{Exists_SD}. Let $K_3'$ be as Proposition \ref{inverse_function} and set
	\[
		K_3=\frac{K_3'R\min\mu(\e)}{10}.
	\]
Set $K=\min\{K_1,K_2,K_3\}$ and assume $\|W\|_S< K$. Let $y$ be as in Corollary \ref{Exists_SD}; that is, $y=x+f=x+\D g$ with $g\in A_R$ and hence $y_\e\in \cS(\Gamma,\mu)$. Furthermore, the joint law of $y$ with respect to $\tau$ is the unique solution to the Schwinger--Dyson equation with potential $V$. This means the map $N\ni \pi(z_\e)\mapsto y_\e\in \cS(\Gamma,\mu)$, $\e\in\vec{E}$, extends to an isomorphism $N\cong C^*(y_\e\colon \e\in\vec{E})$ and
	\[
		\varphi\circ \pi (q) = \tau(q(y)) \qquad \forall q\in B_R.
	\]
In particular, this implies that $\varphi$ is faithful and tracial on $N$. Clearly this isomorphism extends to an isomorphism $W^*(N,\varphi)\cong W^*(y_e\colon e\in \vec{E})$.

We conclude the proof by noting that $y$ generates the same C*-algebra (and hence von Neumann algebra) as $x$. Indeed, recall $f=\D g=\D\Sigma_\mu \hat g$, where $\hat g\in B_R$ is the fixed point of $F$ from Theorem \ref{F_Lipschitz}. In particular, $\|\hat g\|_R \leq 10 \|W\|_R$ and hence
	\[
		\|f\|_R = \|\D\Sigma_\mu \hat g\|_R\leq \frac{1}{R\min\mu(\e)} \|\hat g\|_R\leq \frac{10}{R\min\mu(\e)} \|W\|_R < K_3'.
	\]
Thus, the hypothesis of Proposition \ref{inverse_function} are satisfied and we can write $y=h(x)$ for some $h\in B^{\vec{E}}_{R+K_3'}$. Hence $x$ and $y$ generate the same C*-algebra.
\end{proof}

\begin{rem}
Let $\cP_{\bullet}$ be a finite-depth subfactor planar algebra.  For more details, see \cite{0902.1294}.  In \cite{MR2732052} Guionnet, Jones and Shlyakhtenko have studied the von Neumann algebra $\cM_{0}$ which can be described as
	\[
		\cM_{0} = \overline{\left(\bigoplus_{n \geq 0} \cP_{n}, \Tr\right)}
	\]
with $\Tr$ defined by
	\[
	\Tr(x) = \begin{tikzpicture}[baseline=.5cm]
	\draw (0,0)--(0,.8);
	\nbox{unshaded}{(0,0)}{.4}{0}{0}{$x$}
	\node at (.2, .6) {{\scriptsize{$n$}}};
	\nbox{unshaded}{(0,1.2)}{.4}{.3}{.3}{$\Sigma \, \TL$}	
\end{tikzpicture}\,.
	\]
with $x \in \cP_{n}$ and $\Sigma \, \TL$ is the sum over all loopless Temperley-Lieb diagrams.  It was shown in \cite{MR2807103} that $\cM_{0} = p_{*}\cM(\Gamma, \mu)p_{*}$ where $*$ is the unique depth-zero vertex of $\Gamma$, the principal graph of $\cP_{\bullet}$, and $\mu$ is the associated Perron-Frobenius weighting.  In later work with Zinn-Justin, \cite{MR2989453}, they studied perturbative models of $\cM_{0}$, $\cM^{W}_{0}$ with trace given by
\[
\Tr_{W}(x) = \begin{tikzpicture}[baseline=.5cm]
	\draw (0,0)--(0,.8);
	\nbox{unshaded}{(0,0)}{.4}{0}{0}{$x$}
	\node at (.2, .6) {{\scriptsize{$n$}}};
	\nbox{unshaded}{(0,1.2)}{.4}{.8}{.8}{$\Sigma \, \TL + \D W$}	
\end{tikzpicture}\,.
\]
In \cite{MR3325530} the second author used non-tracial free transport results from \cite{MR3312436} to establish $\cM^{W}_{0} \cong \cM_{0}$ for sufficiently small $W$.   Theorem \ref{thm:iso} yields the same isomorphisms since the projections $p_{v}$ are preserved by the transport maps.
\end{rem}

\section*{Appendix}

\begin{proof}[\textbf{Proof of Proposition \ref{prop_3.2}}]\label{proof_prop_3.2}
Expressing $y$ as $x+f$ and using Lemma \ref{lem:change_of_var}, the Schwinger--Dyson equation becomes
$$
\J^{*}([P + \J f]^{-1}) = M\# (x + f) + (\D W)(x + f).
$$
Then, using $\frac{P}{P+\J f} = P -\frac{\J f}{P+\J f}$ and $\J^*(P)=M\# x$, we obtain
$$
\J^{*}\left(\frac{\J f}{P + \J f}\right) + M \# f + (\D W)(X + f) = 0
$$
Since $\J y = P + \J f$ is invertible, we now apply $(P + \J f) \#$ to both sides.  We get
\begin{align*}
\J^{*}\left(\frac{\J f}{P + \J f}\right)+ (\J f) \# \J^{*}&\left(\frac{\J f}{P + \J f}\right) + M \# f + (\J f) \# (M \# f)\\
&+ (\D W)(x + f) + (\J f) \# (\D W)(x + f) = 0
\end{align*}
It is easy to check that $(\D W)(x + f) + (\J f) \# (\D W)(x + f) = \D(W(x+f))$. Using this and $\frac{P}{P+\J f} = \J f - \frac{\J f^2}{P+\J f}$, we obtain the desired form of the equation. 
\end{proof}

\begin{proof}[\textbf{Proof of Theorem \ref{thm_3.3}}]\label{proof_thm_3.3}
We will show that this holds weakly.  i.e. we show that for all $q \in A^{\vec{E}}$
	\begin{align*}
		\left\langle (\J f) \# \J^{*}((\J f)^{m-1})\right. &- \left.\J^{*}((\J f)^{m}),q\right\rangle= \left\langle \frac{1}{m}\D\left[ (1 \otimes \tau + \tau \otimes 1)\circ\Tr(\J f^m) \right],q\right\rangle.
	\end{align*}
We first examine $\langle (\J f) \# \J^{*}((\J f)^{m-1}),q\rangle$.  This is
\begin{align*}
\sum_{\e\in\vec{E}} \tau([(\J f) \# \J^{*}((\J f)^{m-1})]_{\e}^{*}q_{\e}) &= \sum_{\e, \phi\in\vec{E}}\tau(([\J f]_{\e, \phi}\# \J^{*}((\J f)^{m-1})]_{\phi})^{*}q_{\e})\\
&= \sum_{\e, \phi\in\vec{E}}\tau( [\J^{*}((\J f)^{m-1})]_{\phi}^{*}[\J f]_{\e, \phi} ^{*}\# q_{\e})\\
&= \sum_{\e, \phi\in\vec{E}}\tau( [\J^{*}((\J f)^{m-1})]_{\phi}^{*}[\J f]_{\phi, \e}\# q_{\e})\\ 
&= \sum_{\e, \phi, \omega\in\vec{E}}\tau( [\p_{\omega}^{*}([\J f^{m-1}]_{\phi,\omega})^{*}[\J f]_{\phi, \e}\# q_{\e})\\
&= \sum_{\e, \phi, \omega\in\vec{E}} \tau \otimes \tau\left([\J f^{m-1}]_{\phi, \omega}^{*} \# \p_{\omega}(\p_{\e}(f_{\phi}) \# q_{\e})\right)\\
&= \sum_{\e, \phi, \omega\in\vec{E}} \tau \otimes \tau([\J f^{m-1}]_{\omega,\phi} \# \p_{\omega}(\p_{\e}(f_{\phi}) \# q_{\e}))
\end{align*}
For $a \otimes b \otimes c \in B \otimes B \otimes B$ and $d \in B$, we define $\#_{1}$ and $\#_{2}$ by
$$
(a \otimes b \otimes c) \#_{1} d = adb \otimes c\qquad\text{ and }\qquad (a \otimes b \otimes c) \#_{2} d = a \otimes bdc
$$
From this definition we see that 
$$
\p_{\omega}(\p_{\e}(f_{\phi}) \# q_{\e}) = [(\p_{\omega} \otimes 1)(\p_{\e}(f_{\phi}))] \#_{2}\, q_{\e} + \p_{\e}(f_{\phi}) \# \p_{\omega}(q_{\e}) + [(1 \otimes \p_{\omega})(\p_{\e}(f_{\phi}))] \#_{1} \, q_{\e}
$$
Therefore,
	\begin{align*}
		\langle &(\J f) \# \J^{*}((\J f)^{m-1}),q\rangle = \sum_{\e, \phi, \omega} (\tau \otimes \tau)\left(\vphantom{\sum}[\J f^{m-1}]_{\omega,\phi} \# \p_{\e}(f_{\phi}) \# \p_{\omega}(q_{\e})\right)\\
		&+  \sum_{\e, \phi, \omega\in\vec{E}}  (\tau \otimes \tau)\left( \vphantom{\sum}[\J f^{m-1}]_{\omega,\phi} \# \left\{\vphantom{\sum}[(\p_{\omega} \otimes 1)(\p_{\e}(f_{\phi}))] \#_{2}\, q_{\e} + [(1 \otimes \p_{\omega})(\p_{\e}(f_{\phi}))] \#_{1} \, q_{\e}\right\}\right)
	\end{align*}
We examine the first term on the right hand side.  This is
\begin{align*}
 \sum_{\e, \phi, \omega\in\vec{E}}  (\tau \otimes \tau) \left(\vphantom{\sum}[\J f^{m-1}]_{\omega,\phi} \# [\J f]_{\phi, \e} \# \p_{\omega}(q_{\e})\right) &= \sum_{\e, \omega\in\vec{E}}  (\tau \otimes \tau)\left(\vphantom{\sum} [\J f^{m}]_{\omega, \e} \# \p_{\omega}(q_{\e})\right)\\
 &= \sum_{\e, \omega\in\vec{E}} (\tau \otimes \tau)\left(\vphantom{\sum} [\J f^{m}]^{*}_{\e, \omega} \# \p_{\omega}(q_{\e})\right)\\
 &= \sum_{\e, \omega\in\vec{E}} \tau (\p_{\omega}^{*}([\J f^{m}]_{\e, \omega})^* q_{\e})\\
 &= \langle \J^{*}([\J f]^{m}),q\rangle
\end{align*}
This cancels with the $-\langle \J^{*}([\J f]^{m}),q\rangle$ term above.

We now concentrate on 
	\begin{align}\label{eqn:hash_12_term}
\sum_{\e, \phi, \omega\in\vec{E}}  (\tau \otimes \tau) \left(\vphantom{\sum}[\J f^{m-1}]_{\omega,\phi} \# \left\{\vphantom{\sum}[(\p_{\omega} \otimes 1)(\p_{\e}(f_{\phi}))] \#_{2}\, q_{\e} + [(1 \otimes \p_{\omega})(\p_{\e}(f_{\phi}))] \#_{1} \, q_{\e}\right\}\right)
	\end{align}
Define $N \in \bbM(A_R)$ by
	\[
		\displaystyle [N]_{\phi\omega} = \sum_{\e\in\vec{E}}[(\p_{\omega} \otimes 1)(\p_{\e}(f_{\phi}))] \#_{2}\, q_{\e} + [(1 \otimes \p_{\omega})(\p_{\e}(f_{\phi}))] \#_{1} \, q_{\e}
	\]
Using the tracial property, the expression in (\ref{eqn:hash_12_term}) is equivalent to
$$
\frac{1}{m} \sum_{n=0}^{m-1} (\tau \otimes \tau)\circ\Tr\left(\vphantom{\sum}[\J f]^{n} \# N \# [\J f]^{m-1-n}\right)
$$
We will show that this is the same as $\displaystyle \left\langle \frac{1}{m}\D[(1 \otimes \tau + \tau \otimes 1)\circ\Tr(\J f^{m})],q \right\rangle$ by showing that both are derivatives of the same expression.

For $t \in \R$, define $x^{t}$ by $x_{\e}^{t} = x_{\e} + tq_{\e}$.  Expanding $f$ as a power series and keeping in mind that terms with exactly one ``$t$" survive below, we have:  
	\begin{align*}
		\frac{d}{dt} \bigg|_{t=0} \J f(x^{t})_{\phi, \omega} &= \frac{d}{dt} \bigg|_{t=0} \p_{\omega}f_{\phi}(x^{t})\\
			&= \sum_{\e\in\vec{E}} [(1 \otimes \p_{\e})\p_{\omega}(f_{\phi})] \#_{2} q_{\e} + [(\p_{\e} \otimes 1)\p_{\omega}(f_{\phi})] \#_{1} q_{\e}\\
			&= \sum_{\e\in\vec{E}} [(\p_{\omega} \otimes 1)\p_{\e}(f_{\phi})] \#_{2} q_{\e}+ [(1 \otimes \p_{\omega})\p_{\e}(f_{\phi})] \#_{1} q_{\e}  = [N]_{\phi,\omega}
	\end{align*}
where we used $(1 \otimes \p_{\e})\p_{\omega} = (\p_{\omega} \otimes 1)\p_{\e}$.  This means
	\[
		\frac{1}{m} \frac{d}{dt}  \bigg|_{t=0} (\tau \otimes \tau)\circ\Tr( [\J f(x^{t})]^{m}) = \frac{1}{m} \sum_{n=0}^{m-1} (\tau \otimes \tau)\circ\Tr\left([\J f]^{n} \# N \# [\J f]^{m-1-n}\right)
	\]
Suppose $Q \in \bbM(A_R)$ has only one nonzero entry $a \otimes b$, say in the $\e\phi$ position.  It follows that
\begin{align*}
\frac{d}{dt}  \bigg|_{t=0} (\tau \otimes \tau)\circ\Tr\left( Q(x^{t})\right) &= \delta_{\e,\phi} \frac{d}{dt} \bigg|_{t=0} \tau(a(x^{t}))\tau(b(x^{t}))\\
&= \delta_{\e,\phi}\left( \left[\frac{d}{dt} \bigg|_{t=0} \tau(a(x^{t}))\right]\cdot \tau(b) + \tau(a)\left[\frac{d}{dt} \bigg|_{t=0}\tau(b(x^{t}))\right]\right)\\
&= \delta_{\e,\phi}\sum_{\omega\in\vec{E}} \tau(\D_{\omega^{\op}}(a) q_\omega)\tau(b) + \tau(a)\tau(\D_{\omega^{\op}}(b) q_\omega)\\
&=\delta_{\e,\phi}\sum_{\omega\in\vec{E}} \tau(\D_{\omega^{\op}}[(1\otimes\tau+\tau\otimes 1)(a\otimes b)] q_\omega)\\
&=\delta_{\e,\phi}\sum_{\omega\in\vec{E}} \tau((\D_{\omega}[(1\otimes\tau+\tau\otimes 1)(a^*\otimes b^*)])^* q_\omega)\\
&=\delta_{\e,\phi} \left\langle \D[(1\otimes \tau + \tau\otimes 1)(a^*\otimes b^*)], q\right\rangle\\
&= \left\langle \D[(1\otimes \tau+ \tau\otimes 1)\circ\Tr(Q^*)],q\right\rangle
\end{align*}
By linearity and $(\J f)^*=\J f$, we obtain
\begin{align*}
\frac{1}{m} \frac{d}{dt}  \bigg|_{t=0} (\tau \otimes \tau)\circ\Tr (\J f(x^{t})^{m}) &= \left\langle \frac{1}{m} \D[(1 \otimes \tau + \tau \otimes 1)\circ\Tr(\J f^m)],q\right\rangle\qedhere
\end{align*}
\end{proof}

\begin{proof}[\textbf{Proof of Lemma \ref{lem_3.8}}]\label{proof_lem_3.8}
Recall that for a path $\gamma=\e_1\e_2\cdots \e_n$, we write $x_\gamma$ for a $x_{\e_1}\cdots x_{\e_n}$ and $\mu(\gamma)$ for $\mu(\e_1)+\cdots +\mu(\e_n)$. Suppose $g_{i} = x_{\gamma_{i}}$ with $\gamma_{i}$ a path of length $n_{i}$.  We have:
\begin{align*}
	Q_{m}(\Sigma_{\mu}g_{1},& \dots, \Sigma_{\mu}g_{n})\\
		=& \sum_{\e_{1}, \dots, \e_{m} \in \vec{E}}(1 \otimes \tau + \tau \otimes 1)\circ\Tr\left([\J\D\Sigma_{\mu} g_{1}]_{\e_{1}\e_{2}}  \cdots [\J\D \Sigma_{\mu}g_{m}]_{\e_{m}\e_{1}}\right)\\
		=&  \sum_{\e_{1}, \dots, \e_{m} \in \vec{E}}(1 \otimes \tau + \tau \otimes 1)\circ\Tr\left( \p_{\e_{2}}(\D_{\e_{1}}\Sigma_{\mu}g_{1}) \cdots \p_{\e_{1}}(\D_{\e_{m}} \Sigma_{\mu}g_{m})\right)\\
		=& \sum_{\e_{1}, \dots, \e_{m} \in \vec{E}}\sum_{\gamma_{1} = \gamma_{a}^{(1)}\e_{1}^{\op}\gamma_{b}^{(1)}}\cdots \sum_{\gamma_{m} = \gamma_{a}^{(m)}\e_{m}^{\op}\gamma_{b}^{(m)}}\prod_{i=1}^{m}\frac{1}{\mu(\gamma_i)} \\
		&\times  (1\otimes \tau + \tau\otimes 1)\left(\p_{\e_{2}}(x_{\gamma_{b}^{(1)}\gamma_{a}^{(1)}}) \cdots \p_{\e_{1}}(x_{\gamma_{b}^{(m)}\gamma_{a}^{(m)}})\right)
\end{align*}
Note that as $\e_i$ ranges over $\vec{E}$, there are at most $n_i$ ways to decompose $\gamma_i$ as $\gamma_a^{(i)}\e_i^{\op} \gamma_b^{(i)}$. Then, as $\e_{i+1}$ rangers over $\vec{E}$, the are at most $n_i-1$ terms in $\p_{\e_{i+1}}(x_{\gamma_b^{(i)}\gamma_a^{(i)}})$, each of which is uniquely determined by the degree of the monomial in its first tensor factor. From this, we see that
\begin{align*}
\|Q_{m}&(\Sigma_{\mu}g_{1}, \dots, \Sigma_{\mu}g_{n})\|_{R}\\
&\leq 2\prod_{i=1}^{m} \frac{n_{i}}{\mu(\gamma_{i})} \sum_{k_{1} = 0}^{n_{1}-2}\cdots\sum_{k_{m} = 0}^{n_{m}-2} R^{k_{1}}C^{n_{1}-2-k_{1}}\cdots R^{k_{m}}C^{n_{m}-2-k_{m}}\\
&\leq 2\left(\min_{\e \in \vec{E}} \mu(\e)\right)^{-m} \cdot R^{n_{1} + \cdots + n_{m}}R^{-2m} \sum_{k_{1} = 0}^{n_{1}-2}\cdots\sum_{k_{m} = 0}^{n_{m}-2} \left(\frac{C}{R}\right)^{n_{1} + \cdots + n_{m} - 2m - (k_{1} +  \cdots + k_{m})}\\
&\leq 2\left(\min_{\e \in \vec{E}} \mu(\e)\right)^{-m} \cdot R^{n_{1} + \cdots + n_{m}}R^{-2m}\cdot 2^{m} \qquad\text{ [Using geometric series] }\\
&= 2\left(\frac{2}{R^{2}\min_{\e \in \vec{E}}\mu(\e)}\right)^{m}\prod_{i=1}^{m} \|g_{i}\|_{R}
\end{align*}
For general $g_i\in B_R$, write $\displaystyle g_{i} = \sum_{\gamma} c_{i}(\gamma)x_\gamma$ where the sum is over paths $\gamma$ in $\vec{\Gamma}$.  Note that
$$
\|g_{i}\| = \sum_{\gamma} |c_{i}(\gamma)| R^{\deg{x_\gamma}}.
$$
From the monomial case above, we have
\begin{align*}
\|Q_{m}&(\Sigma_{\mu}g_{1}, \dots, \Sigma_{\mu}g_{m})\|_{R}\\
	 &\leq \sum_{\gamma_1,\ldots, \gamma_{m}} |c_{1}(\gamma_{1})|\cdots |c_{m}(\gamma_{m})| \cdot \|Q_{m}(\Sigma_{\mu}x_{\gamma_{1}}, \dots, \Sigma_{\mu}x_{\gamma_{m}})\|_{R}\\
		&\leq \sum_{\gamma_{1},\ldots, \gamma_{n}} 2\left(\frac{2}{R^{2}\min_{\e \in \vec{E}}\mu(\e)}\right)^{m} |c_{1}(\gamma_{1})|\cdots |c_{m}(\gamma_{m})|  R^{\deg(x_{\gamma_{1}}) + \cdots + \deg(x{\gamma_{m}})}\\
&= 2\left(\frac{2}{R^{2}\min_{\e \in \vec{E}}\mu(\e)}\right)^{m}\prod_{i=1}^{m} \|g_{i}\|_{R}\qedhere
\end{align*}
\end{proof}

\begin{proof}[\textbf{Proof of Lemma \ref{lem_3.9}}]\label{proof_lem_3.9}
For $g,f\in B_R$, we compute using Lemma \ref{lem_3.8}:
\begin{align*}
\|Q&(\Sigma_{\mu}g) - Q(\Sigma_{\mu}f)\|_{R}\\
&\leq \sum_{k=1}^{m} \|Q_{m}(\underbrace{\Sigma_{\mu}g, \dots, \Sigma_{\mu}g}_k, \Sigma_{\mu}f, \dots, \Sigma_{\mu} f) - Q_{m}(\underbrace{\Sigma_{\mu}g, \dots, \Sigma_{\mu}g}_{k-1}, \Sigma_{\mu}f, \dots, \Sigma_{\mu} f)\|_{R}\\
&\leq \sum_{k=0}^{m} 2\left(\frac{2}{R^{2}\min_{\e \in \vec{E}}\mu(\e)}\right)^{m} \|g\|_{R}^{k-1}\|g-f\|_{R}\|f\|_{R}^{m-k}  \qedhere
\end{align*}
\end{proof}

\begin{proof}[\textbf{Proof of Lemma \ref{lem_3.10}}]\label{proof_lem_3.10}
Set $\kappa = \frac{2}{R^{2}\min\mu(\e)}$ and $\|g\|_R = \lambda < \frac{1}{\kappa}$.  Then from Lemma \ref{lem_3.8}, $\|Q_{m}(\Sigma_{\mu}(g))\|_R \leq 2(\kappa\lambda)^{m}$ so it follows the series converges. Note that
\begin{align*}
\|Q&(\Sigma_{\mu}g) - Q(\Sigma_{\mu}f)\|_{R}\\
	&\leq \sum_{m \geq 1}\frac{1}{m}\|Q_{m}(\Sigma_{\mu}g) - Q_{m}(\Sigma_{\mu}f)\|_{R}\\
&\leq \|f - g\|_{R} \sum_{m \geq 1} \sum_{k=0}^{m-1} 2\left(\frac{2}{R^{2}\min\mu(\e)}\right)^{m}\|f\|_{R}^{m-k-1}\|g\|_{R}^{k}\\
&= \|f - g\|_{R} \left(\frac{4}{R^{2}\min\mu(\e)}\right) \sum_{k \geq 0}\sum_{l \geq 0} \left(\frac{2}{R^{2}\min\mu(\e)}\right)^{k}\|f\|_{R}^{k}\left(\frac{2}{R^{2}\min\mu(\e)}\right)^{l}\|g\|_{R}^{l}
\end{align*}
Since $\|f\|_{R}, \, \|g\|_{R} <\displaystyle \frac12 R^{2}\min_{\e\in\vec{E}}\mu(\e)$, we have
\begin{align*}
\|Q(\Sigma_{\mu}f) &- Q(\Sigma_{\mu}g)\|_{R}\\
&\leq \|f - g\|_{R} \left(\frac{4}{R^{2}\min\mu(\e)}\right)\left( \frac{1}{(1 - \frac{2\|f\|_{R}}{R^{2}\min\mu(\e)})(1 - \frac{2\|g\|_{R}}{R^{2}\min\mu(\e)})} \right)
\end{align*}
Plugging in $f = 0$ gives
	\[
		\|Q(\Sigma_{\mu}g)\|_{R} \leq \|g\|_{R} \left(\frac{4}{R^{2}\min\mu(\e)}\right)\left( \frac{1}{1 - \frac{2\|g\|_{R}}{R^{2}\min\mu(\e)}}\right)\qedhere
	\]
\end{proof}

\bibliographystyle{amsalpha}
\bibliography{bibliography}

\end{document}